\documentclass{amsart}
\usepackage{amssymb, amsthm, amsmath, amsfonts,amscd,bm,fancyhdr}
\usepackage{graphics}
\usepackage{hyperref}
\usepackage[all]{xy}
\usepackage{enumerate}
\usepackage[mathscr]{eucal}
\usepackage{bbm}
\usepackage{tikz}
\usetikzlibrary{decorations.pathreplacing,angles,quotes,knots}
\usepackage{parskip}
\usepackage{GWmacros}

\providecommand{\U}[1]{\protect\rule{.1in}{.1in}}
\def\theenumi{\arabic{enumi}}

\def\theenumii{\alph{enumii}}
\def\p@enumii{\theenumi.}

\def\theenumiii{\arabic{enumiii}}
\def\p@enumiii{(\theenumi)(\theenumii)}

\def\p@enumiv{\p@enumiii.\theenumiii}

\theoremstyle{plain}

\newtheorem*{thm}{Theorem}
\newtheorem{lemma}[theorem]{Lemma}

\newtheorem{proposition}[theorem]{Proposition}

\newtheorem{corollary}[theorem]{Corollary}

\numberwithin{equation}{section}

\theoremstyle{definition}

\newtheorem{definition}[theorem]{Definition}
\newtheorem{example}[theorem]{Example}

\newtheorem{remark}[theorem]{Remark}

\newtheorem{thmab}{Theorem}

\newtheorem{corab}[thmab]{Corollary}

\DeclareMathOperator{\FI}{FI}

\DeclareMathOperator{\Geom}{Geom}
\newcommand{\Sn}{\mathfrak{S}}

\newcommand{\Ob}{\mathcal{O}}

\renewcommand{\R}{\mathcal{R}}
\newcommand{\RR}{\mathbb{R}}

\newcommand{\Lap}{\mathcal{L}}
\newcommand{\Gre}{\mathcal{G}}

\newcommand{\X}{\mathcal{X}}

\newcommand{\dt}{\bullet}
\newcommand{\fg}{$\FI$--graph}
\newcommand{\fs}{$\FI$--set}
\newcommand{\arXiv}[1]{\href{http://arxiv.org/abs/#1}{\nolinkurl{arXiv:#1}}}
\newcommand{\arXivV}[2]{\href{http://arxiv.org/abs/#1}{\nolinkurl{arXiv:#1v#2}}}

\title{Families of Markov chains with compatible symmetric-group actions}

\author[E.~Ramos]{Eric Ramos}
\address[E. Ramos]{University of Oregon Department of Mathematics, Fenton Hall, Eugene, OR 97405}
\email{eramos@uoregon.edu}

\author[G.~White]{Graham White}
\address[G. White]{Indiana University --- Bloomington Department of Mathematics, Rawles Hall, Bloomington, IN 47405}
\email{grrwhite@iu.edu}

\thanks{The first author was supported by NSF grant DMS-1704811.}
\begin{document}

\begin{abstract}
For each $n,r \geq 0$, let $KG(n,r)$ denote the Kneser Graph; that whose vertices are labeled by $r$-element subsets of $n$, and whose edges indicate that the corresponding subsets are disjoint. Fixing $r$ and allowing $n$ to vary, one obtains a family of nested graphs, each equipped with a natural action by a symmetric group $\Sn_n$, such that these actions are compatible. Collections of graphs of this type are common in algebraic combinatorics and include families such as the Johnson Graphs, Crown Graphs and Rook Graphs. In previous work \cite{RW}, the authors systematically studied families of this type using the language of representation stability and $\FI$-modules. In that work, it is shown that such families of graphs exhibit a large variety of asymptotic regular behaviors.\\

The present work applies the theory developed in \cite{RW}, later refined in \cite{RSW}, to study random walks on the graphs of such families. We show that the moments of hitting times exhibit rational function behavior asymptotically. By consequence we conclude similar facts about the entries of the discrete Green's functions, as defined by Chung and Yau \cite{CY}. Finally, we illustrate how the algebro-combinatorial structure of the graphs in these families give bounds on the mixing times of random walks on those graphs.
We suggest some possible directions for future study, including of the appearance, or not, of the cut-off phenomenon, originally presented by Diaconis \cite{D}.
\end{abstract}


\maketitle

\section{Introduction}
\subsection{Motivation}
For each $r,n \geq 0$, let $KG(n,r)$ denote the graph whose vertices are labeled by the $r$-element subsets of $[n] = \{1,\ldots,n\}$, and whose edges indicate that the corresponding sets are disjoint. In the literature, these graphs are called \emph{Kneser graphs}. What is relevant for the present work is the following structure that one may put on $\{KG(n,r)\}_{n \geq 0}$, for any fixed $r$: for every injection of sets $f:[n] \hookrightarrow [m]$, one obtains a graph homomorphism, (an adjacency-preserving map of vertices),
\[
KG(f):KG(n,r) \rightarrow KG(m,r).
\]
Restricting to the cases wherein $f$ is a permutation, we see that the family $\{KG(n,r)\}$ can be thought of as a family of nested graphs, each of which carries the action of a symmetric group in such a way that these actions are compatible with one another. Families of this type are ubiquitous throughout algebraic combinatorics (see Example \ref{smallex}, as well as the numerous other examples throughout this work). In \cite{RW}, the authors developed a framework for studying these examples by using the language of \emph{representation stability} and \emph{$\FI$-modules} \cite{CEF}.

Let $\FI$ denote the category of finite sets and injective maps. Then an \emph{$\FI$-graph} $G_\dt$ is a functor from $\FI$ to the category of graphs. We have already seen that the assignment
\[
[n] \mapsto KG(n,r) =: G_n
\]
is an $\FI$-graph for each fixed $r \geq 0$. Furthermore, this particular $\FI$-graph is \emph{finitely generated}, in the sense that for $n \gg 0$ every vertex of $G_{n+1}$ is in the image of some $G_n$ (see Definition \ref{fingen}). In \cite{RW}, it is shown that the condition of finite generation yields a plethora of important structural properties of the graphs $G_n$. For instance, one has the following.

\begin{thm}[Ramos \& White, \cite{RW}]
Let $G_\dt$ be a finitely generated $\FI$-graph. Then for each $r \geq 0$, and each $n \gg 0$, the number of walks in $G_n$ of length $r$ agrees with a polynomial.
\end{thm}

Seeing such asymptotic regularity in the total count of walks of a given length, one might be tempted to ask the following somewhat vague question: \emph{do the standard statistics of random walks on $G_n$ exhibit similar regularity as $n \rightarrow \infty$?} The purpose of this work is to make this question precise, and attempt to answer it.

\subsection{Hitting times}
Let $X_t$ denote a \emph{Markov chain} on a state space $\mathcal{X}$ (see Definition \ref{markovdef}). For the sake of concreteness, we want to think about $\mathcal{X}$ as being the vertex set of some graph, while $X_t$ is a random walk on this graph (see Definition \ref{graphdef}). Then given any two elements $x$ and $y$ of $\mathcal{X}$, the \emph{hitting time} from $x$ to $y$ is denoted by $\tau_{x,y}$. This is the random variable whose value is the number of steps taken for a random walk beginning at $x$ to reach $y$. Understanding the behavior of $\tau_{x,y}$ is an important problem in the study of Markov chains --- see for instance Chapter $10$ of \cite{LPW}.

Returning to the context of the previous section, we first need to make precise what we mean by hitting times on an $\FI$-graph $G_\dt$. Let $m \gg 0$ be fixed, and let $x,y$ be two vertices of $G_m$. Then for each $n \geq m$, we define $x(n)$ and $y(n)$ to be the vertices
\[
x(n) := G(\iota_{m,n})(x), y(n) := G(\iota_{m,n})(y),
\]
where $\iota_{m,n}:[m] \hookrightarrow [n]$ is the standard inclusion of sets, and $G(\iota_{m,n}):G_m \rightarrow G_n$ is the induced map. For each $n \geq 0$, let $X_{t,n}$ denote the Markov chain modeling the simple random walk on $G_n$. Then we obtain a sequence of random variables by setting,
\[
\tau_{x,y}(n) := \tau_{x(n),y(n)}
\]
Concretely, we think of $x(n)$ and $y(n)$ as being the ``same'' vertices as $x$ and $y$, respectively, just living in a bigger graph containing $G_m$. With this notation, the following will be proven during the course of this work.

\begin{thmab}
Let $G_\dt$ denote a finitely generated $\FI$-graph, let $X_{t,\dt}$ denote the family of Markov chains modeling the simple random walk on $G_\dt$, and let $x,y$ be vertices of $G_m$ for some $m \gg 0$. Then for all $n \geq m$, and all $i \geq 1$ the function
\[
n \mapsto \mu_i(\tau_{x,y}(n))
\]
agrees with a rational function, where $\mu_i$ is any of the $i$-th moment, the $i$-th central moment, or the $i$-th cumulent.
\end{thmab}

The simplest case of the above theorem implies that the expected value of the hitting time random variable between two vertices is eventually equal to a rational function. This case is particularly relevant, as it implies similar conclusions about the so-called \textbf{discrete Green's functions}.

In \cite{CY}, Chung and Yau introduce what they call the discrete Green's functions of a graph. Pulling inspiration from the more classical setting of partial differential equations, the discrete Green's function can be thought of as a partial inverse to the Laplacian of the graph (see Definition \ref{laplace}). Since their introduction, discrete Green's functions have found a variety of applications through algebraic combinatorics (see \cite{E} or \cite{XY}, for instance). One recurring themes in these works, however, is that they are surprisingly difficult to compute. While our work does not explicitly compute any Green's functions, we can use known connections between their values and hitting times to prove the following.

\begin{corab}
Let $G_\dt$ be a finitely generated $\FI$-graph, and for each $n$ let $\Gre_n$ denote the discrete Green's function of $G_n$. Then for any vertices $x,y$ of $G_m$, with $m \gg 0$, and any $n \geq 0$, the function
\[
n \mapsto \Gre_n(x(n),y(n))
\]
agrees with a function of finite degree over $\Q(n)$.
\end{corab}

Note that the word degree in the above corollary is used in the Galois-theoretic sense. In fact, to refine the above, we can think of this function as being some algebraic combination of rational functions and square roots of rational functions.

\subsection{Mixing times}

Another natural question to ask about a Markov chain is what happens when it is run for a long time. Under mild conditions on the chain, the answer is that it approaches a unique stationary distribution. Thus, the question becomes ``How fast does the chain approach the stationary distribution?''. The \emph{mixing time} is a measure of this time taken --- see Definition \ref{mixdef} and surrounding parts of Section \ref{sec:markovbackground}. Essentially, the mixing time is the time taken for the Markov chain in question to get within a certain distance of its stationary distribution, from any starting state. More details on mixing times can be found in, for example, \cite{LPW}.

Given a family of Markov chains indexed by $n$, one may ask how the mixing times depend on $n$. In this paper, we will be interested in random walks on \fg s. If $G_\dt$ is an \fg, then we will examine how the mixing times of various random walks on $G_n$ depends on $n$. We will find that for a certain family of walks where the transition probabilities depend on the orbit sizes in an appropriate way (Definition \ref{def:weightedwalk}), the mixing time is constant in $n$.

\begin{thmab}
For any finitely generated \fg{} which is eventually not bipartite, the weighted walk on $G_n$ has mixing time which is constant in $n$, in the sense of Remark \ref{rem:mixingeps}.
\end{thmab}

We will then see that for walks with other transition probabilities, the mixing times are bounded by how different these transition probabilities are from those in the previous case. (More precisely, the ratio in question is between probability flows $\pi(x)P(x,y)$, not just the transition probabilities $P(x,y)$).

\begin{thmab}
If $G_\dt$ is a finitely generated \fg{} which is eventually not bipartite, then a general model of a random walk on $G_n$ has mixing time bounded above by a constant times $\rho(n)^{-1}$ (Definition \ref{def:rho}). 
\end{thmab}

\subsection{The scope of this work}

While the introduction has been written in the language of $\FI$-graphs for concreteness, we take a moment here to note that most of the work of this paper will actually be done in a much more general context.

In \cite{RSW}, Speyer and the authors lay the foundation for the study of $\FI$-sets, functors from $\FI$ to the category of finite sets, equipped with relations (see Definition \ref{fiset} and Definition \ref{relation}). $\FI$-graphs fall into this more general framework, as pair of a vertex $\FI$-set and an edge relation, but they are far from the only types of objects that do. For instance, the theory of $\FI$-posets, proposed by Gadish \cite{G}, are also examples.

Our results on hitting times will apply to any reasonable random walk on a finitely generated $\FI$-set (see Definition \ref{virtualrelation} for what is meant here by reasonable). For instance, if one would instead prefer to work with a lazy random walk on an $\FI$-graph, the same results apply. Our results on mixing times are more limited in scope. In particular, they will only apply to random walks on (undirected) $\FI$-graphs, as \emph{reversibility} of the relevant Markov chains plays a key role in the proof (see Definition \ref{reverse}).

\subsection{Interesting open questions}

While this work hopes to lay the foundation for a systematic study of random walks on $\FI$-sets, there is still much that is not well understood. For instance, can anything be said about the relative sizes of the moments of the hitting time random variables? One might hope that the standard deviation would be small relative to the expected value, to obtain a concentration result. We prove that the various moments share a kind of asymptotic regularity, but our methods do not give information on their relative sizes.

Another possible direction of future study is the question of \emph{cut-off}. In his seminal work \cite{D}, Diaconis defined the cutoff phenomenon for families of Markov chains. A family $X_{t,n}$ of Markov chains is said to have cutoff if, for $n \gg 0$, the chain quickly moves from being very far from mixed to very close to mixed. While it is believed that cutoff is common in many natural families of Markov chains, this is notoriously difficult to prove. In many examples where one can actually compute mixing times, random walks on finitely generated $\FI$-graphs exhibit cutoff. However, one can also construct examples where the walk does not have cutoff (see Example \ref{nocutoff}). It would be interesting to know whether one can put natural conditions on the $\FI$-graph so that these random walks have cutoffs.

One reason that we believe that such a condition should exist is related to the main theorem of Speyer and the authors from \cite{RSW}. One consequence of this theorem is that the probability transition matrix of a random walk on a finitely generated $\FI$-graph has very restrictive behavior in its spectrum. Namely, in the case where $G_n$ has a single vertex orbit with respect to the $\Sn_n$-action, it can be shown that the second biggest eigenvalue of the probability transition matrix has multiplicity growing with $n$. Heuristics on the cutoff phenomenon due to Diaconis \cite{D} suggest that these random walks should be expected to have cutoff.

\section*{Acknowledgments}
The authors would like to send their sincere thanks to Thomas Church, Jennifer Wilson, and Zachary Hamaker for useful input during the early stages of this work. The first author was supported by NSF grant DMS-170481

\section{Background}

\subsection{Markov chains and random walks}
\label{sec:markovbackground}

We begin with a review of some theory of Markov chains. Special emphasis will be placed on the notions of hitting time (see Definition \ref{hitdef}) and mixing time (see Definition \ref{mixdef}). The example of random walks on graphs will appear beginning in the next section both as motivation and as a way to ground the material. All of what follows can be found in any standard text on the subject.

\begin{definition}\label{markovdef}
Let $\mathcal{X}$ be a finite set. Then a \textbf{Markov chain} on $\mathcal{X}$ is a family of random variables $\{X_t\}_{t = 0}^\infty$ such that for all $t \geq 0$, and all $(t+1)$--tuples $(x_0,\ldots,x_t) \in \mathcal{X}^{t+1}$,
\begin{enumerate}
\item $\mathbb{P}(X_t = x_t \mid X_{t-1} = x_{t-1}, \ldots , X_{0} = x_0) = \mathbb{P}(X_t = x_t \mid X_{t-1} = x_{t-1})$, and
\item $\mathbb{P}(X_t = x_t \mid X_{t-1} = x_{t-1}) = \mathbb{P}(X_{t-1} = x_t \mid X_{t-2} = x_{t-1}).$
\end{enumerate}
That is, a Markov chain on $\mathcal{X}$ is a procedure for randomly moving between elements of the \textbf{state space} $\mathcal{X}$, such that at each time the probability of moving to any element of $\mathcal{X}$ depends only on the current state. 

The information necessary to define a Markov chain is the state space $\X$ and the collection of transition probabilities --- the probabilities of moving from any state to any other. These probabilities are collected in the \textbf{transition matrix}, whose $(i,j)$--entry is the probability of moving from state $i$ to state $j$ in a single step. If $a,b \in \X$ are such that $P(a,b) > 0$, then we say that $b$ is a \textbf{neighbor} of $a$.

The transition matrix can also be seen as the endomorphism $P:\RR \mathcal{X} \rightarrow \RR \mathcal{X}$ of the $\RR$-linearization of $\mathcal{X}$ given entry-wise by
\[
P(x,y) = \mathbb{P}(X_t = y \mid X_{t-1} = x).
\]
We say that a Markov chain $\{X_t\}_t$ on $\mathcal{X}$ is \textbf{connected} or \textbf{irreducible} if for any pair of states $x,y \in \mathcal{X}$ there is some $t > 0$ such that
\[
P^t(x,y) > 0
\]
The matrix $P$ is independent of the choice of initial distribution $\mathbb{P}(x) := \mathbb{P}(X_0 = x)$. We will usually interpret a choice of initial distribution as a row vector in $\RR \mathcal{X}$ whose coordinates sum to 1. A \textbf{stationary distribution of a Markov chain} is a choice of initial distribution $\pi$ having the property that $\pi \cdot P = \pi$.

When the state space $\mathcal{X}$ is understood, we will often write $(X_t,P)$ to denote the pair of a Markov chain and its associated transition matrix.
\end{definition}

One natural question that one might ask about a connected Markov chain $(X_t,P)$ on a state space $\mathcal{X}$ is the time the chain will take to move between two chosen points. This quantity is a hitting time. That is, if the Markov chain is started in state $x$, then consider the first time at which it moves to state $y$. This random variable is the hitting time of $y$ from $x$. If the Markov chain is connected and finite, then the hitting times are almost surely finite. 

\begin{definition}\label{hitdef}
Let $(X_t,P)$ be a connected Markov chain on a state space $\mathcal{X}$. Then for any $x,y \in \mathcal{X}$, we define the \textbf{hitting time from $x$ to $y$} as
\[
\tau_{x,y} := \min\{t \mid X_t = y, X_0 = x\}.
\]
We will often work with the expected value
\[
Q(x,y) := \mathbb{E}(\tau_y \mid X_0 = x)
\]
\end{definition}

One useful fact for computing hitting times is the following recursive formulation, derived by considering the various possibilities for the first step from the state $x$.

\begin{lemma}\label{hitrecursion}
Let $(X_t,P)$ be a connected Markov chain on a state space $\mathcal{X}$. Then for any distinct $x,y \in \mathcal{X}$ one has,
\[
Q(x,y) = 1 + \sum_{z \in \mathcal{X}} P(x,z)Q(z,y)
\]
\end{lemma}

A natural question about Markov chains is what their long-term behavior is like. For nice enough chains, the answer is 
that they get closer and closer to their stationary distribution. We will now make this statement precise. 

We will later discuss the mixing time of a Markov chain, which relates to when $t$ is large enough so that the probability distribution associated to $X_t$ is sufficiently similar to the stationary distribution (see Definition \ref{mixdef}). The following theorem guarantees that not only does the stationary distribution exist, it is also necessarily unique. 

\begin{theorem}[Proposition 1.14 and Corollary 1.17 of \cite{LPW}]
Let $(X_t,P)$ be a connected Markov chain on a state space $\mathcal{X}$. Then there exists a unique distribution $\pi$ such that $\pi \cdot P = \pi$.
\end{theorem}

\begin{remark}\label{righteigen}
Because it will be useful later, we note that standard facts in linear algebra imply that if $(X_t,P)$ is a connected Markov chain, then the subspace generated by $\mathbf{1} := \sum_{x \in \mathcal{X}} x$ is the (right) eigenspace of $P$ associated to the eigenvalue 1. In particular, all (right) eigenvectors of $P$ associated to 1 are supported on every element of $\mathcal{X}$.
\end{remark}

\begin{definition}\label{reverse}
If $(X_t,P)$ is a connected Markov chain on a state space $\mathcal{X}$ with stationary distribution $\pi$, then we say $(X_t,P)$ is \textbf{reversible} if for all $x,y \in \mathcal{X}$
\[
\pi(x)P(x,y) = \pi(y)P(y,x).
\]
\end{definition}

Reversibility essentially entails that a Markov chain can be reversed and remain a Markov chain. This condition will be useful to us when studying random walks on graphs (see Definition \ref{graphdef}). Reversibility may also be seen as saying that if a step is taken from the stationary distribution $\pi$, then for each edge $(x,y)$, equal probability `flows' from $x$ to $y$ and from $y$ to $x$. 

In order to talk about Markov chains approaching their stationary distributions, we will need to be able to measure the distance between distributions. 

\begin{definition}
If $\mu$ and $\nu$ are two probability distributions on a set $\mathcal{X}$, then the \textbf{total variation distance} between $\mu$ and $\nu$ is the maximum value of $\mu(A) - \nu(A)$ over all events $A \subseteq \mathcal{X}$. Equivalently (for the finite chains we will consider), it is equal to the sum $$\sum_{x \in \mathcal{X}} \frac{1}{2}\left|\mu(x) - \nu(x)\right|.$$  
\end{definition}

The total variation distance is essentially the $L_1$ distance between two probability distributions. One can also use other measures of distance. For our purposes, the choice of distance function is not terribly important. 

\begin{Theorem}[Theorem 4.9 of \cite{LPW}]
Let $P$ be a Markov chain which is irreducible and aperiodic, with stationary distribution $\pi$. Then there exist constants $\al \in (0,1)$ and $C > 0$ so that for any starting state and any time $t$, the distance of the distribution after $t$ steps of $P$ from the stationary distribution $\pi$ is at most $C\al^t$. 
\end{Theorem}

This theorem requires that the Markov chain in question be irreducible --- that it is possible to get from any state to any other state, and also that it be aperiodic --- that it is not the case that all paths from a state to itself have even length, or length a multiple of any other period. 

\begin{definition} \label{mixdef}
Let $P$ be an irreducible and aperiodic Markov chain on the state space $\mathcal{X}$, and $\eps$ be any positive constant. The \textbf{mixing time} $\tme$ is the smallest time so that for any starting state $x \in \mathcal{X}$, the distribution after $\tme$ steps is within $\eps$ of the stationary distribution $\pi$. 
\end{definition}

Because the distance from stationarity decays exponentially, the mixing times $t_{\text{mix}(\eps_1)}$ and $t_{\text{mix}(\eps_2)}$ have the same order for any $\eps_1$ and $\eps_2$ which are both less than $\frac{1}{2}$. That is, their ratio depends only on the values of $\eps_1$ and $\eps_2$, not on other parameters. This means that it makes sense for us to claim that a family of Markov chains indexed by $n$ have mixing times which are constant in $n$, or at most quadratic in $n$, for example, without needing to specify $\eps$. 

\begin{Remark}
\label{rem:mixingeps}
Given a family of Markov chains indexed by $n$, we will sometimes want to say things like `These chains mix in a single step', or `These chains mix in five steps'. Statements like these should be understood to mean that for any $\eps$, there exists $N$ so that for all $n > N$, the claimed bound is true of $\tme$.  
\end{Remark}

Our main tool for bounding mixing times will be coupling.

\begin{Definition}
A coupling of two Markov chains is a probability distribution on pairs of evolutions of each chain, such that the marginal distributions are the same as the distributions of the original chains. That is, we run a copy of each chain, and are allowed to introduce arbitrary correlations between the two. 
\end{Definition}

The explicit couplings constructed in Section \ref{sec:mixingpf} will mostly be Markovian, which means that what one chain does between time $t-1$ and $t$ depends only on what the other chain is doing at the same time, not on the history of the other chain.  

Given a Markov chain, any coupling between two copies of the chain gives a bound on the mixing time. The bound depends on the time taken for the two chains to be in the same state as one another (and will only be useful if the coupling is designed to achieve this). For any such coupling we can (and will) decree that once the two chains are in the same state, they will move in the same way.

\begin{Theorem}[Theorem 5.4 of \cite{LPW}]
\label{the:couplingmixing}
Let $P$ be a Markov chain on the state space $\mathcal{X}$, and $\eps$ be any positive constant. If there is a coupling of two copies of $P$ so that when two chains are started in any two states and run according to this coupling, after $t$ steps there is at least a probability of $(1 - \eps)$ that the two chains are in the same state, then the distribution after $t$ steps from any starting position is within $\eps$ of the stationary distribution $\pi$. Therefore the mixing time $\tme$ is at most $t$. 
\end{Theorem}

Thus, when we want to bound the mixing time of random walks on \fg s in Section \ref{sec:mixingpf}, we will construct couplings between two random walks on the same graphs, aiming for them to meet as soon as possible and thence move together.

\subsection{Discrete Green's functions}
One of the most interesting class of examples of Markov chains are random walks on graphs. In this paper we will use random walks on $\FI$-graphs as guiding examples for some of the more abstract results we obtain.

We begin this section by establishing the graph theoretic notation to be used in the remainder of the paper. This notation is mostly chosen to be consistent with \cite{CY}.

\begin{definition}
A \textbf{graph} is a pair $G = (V,E)$, where $V$ is a finite set of \textbf{vertices} $V$, paired with a finite set of \textbf{edges} $E \subseteq \binom{V}{2} \cup V$. Those edges $e \in E \cap V$ are known as \textbf{loops}. A graph will be called \textbf{simple} if it does not have any loops. In situations where we are concerned with more than a single graph, we will often write $V(G)$ and $E(G)$ to denote the vertex and edge sets of $G$, respectively. The \textbf{endpoints} of an edge $e \in E(G)$ are the vertices of $G$ which define $e$; either a pair of vertices or a single vertex in the case where $e$ is a loop.

\emph{In this paper we will assume without further notice that all graphs are connected in the usual sense.}

The \textbf{degree of a vertex} is defined to be
\[
\mu(v) := |\{e \in E(G) - V(G) \mid \text{ $v$ is an endpoint of $e$}\}| + 2|\{E(G) \cap V(G)\}|
\]

A \textbf{homomorphism} of graphs $\phi:G \rightarrow G'$ is a map between vertex sets $\phi:V \rightarrow V'$ such that if $e \in E(G)$ is an edge of $G$, then $\phi(e)$ is an edge of $G'$.

A \textbf{subgraph} of a graph $G$ is a graph $H$ such that $V(H) \subseteq V(G)$ and $E(H) \subseteq E(G)$. We say that $H$ is an \textbf{induced subgraph} of $G$ if it is a subgraph such that if $e \in E(G)$ is any edge of $G$ with endpoints in $V(H)$, then $e \in E(H)$.
\end{definition}

\begin{remark}
It will also be convenient to think of a graph as a pair of a finite set $V$ with a symmetric relation $E \subseteq V \times V$. We will use these two descriptions of a graph interchangeably in what follows.

Further note that our definition of graph permits loops. In the context of this paper, this is done as a convenience so that one can consider random walks wherein it is possible to stay fixed at a vertex (see Definition \ref{graphdef}). We will usually not meaningfully distinguish a graph from that same graph with loops added or removed.
\end{remark}

\begin{example}\label{smallex}
Some examples of graphs that we will see throughout this work include:
\begin{itemize}
\item the complete graphs $K_n$. These are those graphs with $V(K_n) = [n]$ and edge set $E(K_n) = \binom{[n]}{2}$;
\item the Kneser graphs $KG(n,r)$. These are those graphs with $V(KG(n,r)) = \binom{[n]}{r}$ and whose edge relation indicates the two chosen subsets are disjoint;
\item the Johnson graphs $J(n,r)$. These are those graphs with $V(J(n,r)) = \binom{[n]}{r}$ and whose edge relation indicates the two chosen subsets have intersection of size $r-1$;
\item Crown graphs $C(n,r)$. These are those graphs with $V(C(n,r)) = [n]  \bigsqcup [n]$ and whose edge relation is given by $\{((i,1),(j,2)), ((i,2),(j,1)) \mid i \neq j\}$.
\end{itemize}
\end{example}

Many of the techniques used in this paper can be said to live at the interface of graph theory and algebra. To accomplish this fusion, we often associate to graphs certain vector spaces which are designed to encode combinatorial invariants. For now, we will be concerned with two particular examples of this perspective.

\begin{definition}\label{laplace}
A \textbf{weighted graph} is a graph $G = (V,E)$ paired with a function $w:V \times V \rightarrow \RR_{\geq 0}$ such that for all $x \neq y \in V$:
\begin{enumerate}
\item $w_{x,y} = w_{y,x}$;
\item $w_{x,y} = 0$ if and only if $\{x,y\} \notin E$.
\end{enumerate}
Given a weighted graph, we define the \textbf{weighted degree} of a vertex to be
\[
d_x := \sum_{y \in V} w_{x,y}.
\]
The \textbf{adjacency matrix} of a weighted graph is the endomorphism $A_G:\RR V \rightarrow \RR V$ defined entry-wise by
\[
A_G(x,y) = w_{x,y}
\]
The \textbf{Laplacian matrix} of a weighted graph is the endomorphism $L_G:\RR V \rightarrow \RR V$ defined entry-wise by
\[
L_G(x,y) = \begin{cases} d_x - w_{x,x} &\text{ if $x = y$}\\ -w_{x,y} &\text{ otherwise.}\end{cases}
\]
\end{definition}

\begin{remark}
It is important to note that the weighted degree of a vertex is not necessarily the same as the degree of that vertex. The former depends on the weight function, while the latter is an invariant of the graph itself.
\end{remark}

Weighted graphs can be thought of as encoding certain statistics of random walks on the graph.

\begin{definition}\label{graphdef}
Let $G$ be a graph with vertex set $V$ and edge set $E$. Then a \textbf{model of a random walk on $G$} is a connected, reversible Markov chain $(X_t,P)$ with state space $V$, such that for all $x \neq y \in V$, $P(x,y) = 0$ whenever $\{x,y\} \notin E$.
\end{definition}

\begin{example}
Given a graph $G = (V,E)$, we can consider $G$ as a weighted graph by setting
\[
w_{x,y} = \begin{cases} 1 &\text{ if $\{x,y\} \in E$}\\ 2 & \text{if $x = y \in E$}\\ 0 &\text{ otherwise.}\end{cases}
\]
Note that any time we speak of the adjacency matrix or Laplacian of a graph, these are the weights we are implicitly using.

If we are given the model of a random walk on $G$, $(X_t,P)$, then we may define a weight function on $G$ in terms of $P$ and the stationary distribution $\pi$ by setting
\[
w_{x,y} = P(x,y)\pi(x).
\]
It is clear that this function satisfies the second condition required in the definition of a weight function. The fact that it satisfies the first follows from the fact that the Markov chain $X_t$ is reversible. If a weighted graph $G$ has the property that $w_{x,y} = P(x,y)\pi(x)$ for some model of a random walk on $G$, then we say $(G,w)$ is \textbf{stochastic}.
\end{example}

\begin{remark}\label{allstoch}
We note that any weighted graph can be scaled to be stochastic. This is achieved by defining a new weight function
\[
\widetilde{w}_{x,y} = \frac{w_{x,y}}{\sum_{s,t}w_{s,t}}
\]
\end{remark}

Some of the main results of this work are concerned with understanding the hitting times of certain families of Markov chains. As an application of these results, we will also be able to conclude facts about discrete Green's functions associated to natural families of graphs. Discrete Green's functions in this context were introduced by Chung and Yau in \cite{CY}. Following that initial work, Green's functions were used by Ellis \cite{E}, Xu and Yau \cite{XY}, and others to prove non-trivial facts about a variety of combinatorial games on graphs.

\begin{definition}
Let $G$ be a weighted graph with weight function $w$. Then the \textbf{normalized Laplacian} of $G$ is the linear function $\Lap: \RR V(G) \rightarrow \RR V(G)$ defined by
\[
\Lap = T^{-\frac{1}{2}}L_GT^{-\frac{1}{2}}
\]
where $T$ is the diagonal matrix with entries
\[
T(x,x) = d_x
\]

We observe that the kernel of $\Lap$ has dimension one, and is spanned by $\phi_0 = \sum_{x \in V} \frac{d_x}{vol} x$, where
\[
vol = \sum_x d_x.
\]
The \textbf{discrete Green's function} associated to $\Lap$ is the (unique) linear function $\Gre$ satisfying:
\begin{enumerate}
\item $\Gre \Lap = \Lap \Gre = Id - P_0$, where $P_0$ is the projection onto the vector $\phi_0$;
\item $\Gre P_0 = 0$.
\end{enumerate}
We therefore think of $\Gre$ as a kind of quasi-inverse to the normalized Laplacian.
\end{definition}

\begin{remark}
The techniques used in the later sections of this paper can be used to prove facts about random walks on certain families of \textbf{directed graphs}. Green's functions as we have presented them above, and as they were originally introduced in \cite{CY}, are only defined on undirected graphs. 
\end{remark}

While it is somewhat more typical for one to use Green's functions to prove statements about important combinatorial invariants, such as hitting times, in this work we take the opposite approach. In Section \ref{hittingtime} we prove stability results about hitting times of random walks on $\FI$-sets (see Definition \ref{fiset}). Using these results we will be able to prove analogous stability theorems for Green's functions on $\FI$-graphs (see Example \ref{figraph}). The main tool we use to achieve this connection is the following.

\begin{theorem}[Chung and Yau, Theorem 8 of \cite{CY}]\label{hittinggreen}
Let $G = (V,E)$ be a stochastic weighted graph, and for any $x,y \in V$ write $Q(x,y)$ for the expected hitting time from $x$ to $y$. Then,
\[
\Gre(x,y) = \frac{\sqrt{d_xd_y}}{vol}\left( Q(x,y) - \frac{1}{vol}\sum_{z \in V} d_zQ(z,y)\right),
\]
where $vol = \sum_{z \in V} d_z$.
\end{theorem}

\subsection{$\FI$-sets and relations}

In this section we review the theory of FI-sets and relations first explored by the authors and Speyer in \cite{RSW}. Along the way we will also give a very brief overview of the theory of $\FI$-modules, discussed in \cite{CEF,CEFN}, and $\FI$-graphs, discussed in \cite{RW}.

\begin{definition}\label{fiset}
We write $\FI$ to denote the category whose objects are the sets $[n] = \{1,\ldots,n\}$, and whose morphisms are injective maps of sets. An \textbf{$\FI$--set} is a functor $Z_\dt$ from $\FI$ to the category of finite sets. If $Z_\dt$ is a $\FI$--set, and $n$ is a non-negative integer, we write $Z_n$ for its evaluation at $[n]$. If $f:[n] \hookrightarrow [m]$ is an injection of sets, then we write $Z(f)$ for the map induced by $Z_\dt$.

An \textbf{$\FI$-subset}, or just a \textbf{subset}, of an $\FI$-set $Z_\dt$ is an $\FI$-set $Y_\dt$ for which there exists a natural transformation $Y_\dt \rightarrow Z_\dt$ such that $Y_n \hookrightarrow Z_n$ is an injection for all $n \geq 0$.
\end{definition}

While the above definition might appear somewhat abstract, one thing we hope to impress upon the reader is that one can think about these objects in quite concrete terms. To see this, first observe that for each $n$, $Z_n$ carries the natural structure of an $\Sn_n$-set, induced from the endomorphisms of $\FI$. With this in mind, one may therefore think of an $\FI$-set $Z_\dt$ as a sequence of $\Sn_n$-sets $Z_n$, which are compatible with one another according to the actions of the morphisms of $\FI$.

\begin{example} \label{firstex}
Perhaps the most elementary example of an $\FI$-set is that defined by the assignment
\[
Z_n := [n].
\]

For a more interesting example, let $\lambda = (\lambda_1,\ldots,\lambda_h)$ denote a partition of some fixed integer $m$. Then we may associate to $\lambda$ a conjugacy class of $\Sn_m$; the class associated to the cycle structure $(\lambda_1,\ldots,\lambda_h)$. We write $c_{\lambda}$  to denote this conjugacy class. For each $n \geq m$, we define $\lambda[n]$ to be the partition
\[
\lambda[n] := (\lambda,\underbrace{1,\ldots,1}_{\text{ $n-m$ times}}),
\]
and set,
\[
c^n_\lambda = \text{ the conjugacy class of $\Sn_n$ associated to $\lambda[n]$.}
\]
Then we obtain an $\FI$ set whose assignment on points is given by
\[
Z_n = \begin{cases} \emptyset &\text{ if $n < m$}\\ c_\lambda^n & \text{ otherwise.}\end{cases}
\]
For any injection of sets $f:[n] \hookrightarrow [r]$, and any $\pi \in c_\lambda^n$ we define
\[
Z(f)(\pi)(i) = \begin{cases} i &\text{ if $i$ is not in the image of $f$}\\ f \circ \pi \circ f^{-1}(i) &\text{ otherwise.}\end{cases}
\]
In short, we see that this $\FI$-set encodes and extends on the conjugation actions on the classes $c_\lambda^n$. This particular $\FI$-set was studied by the first author in \cite{R}, due to its connection with quandle theory.
\end{example}

As one might expect, it is in the best interest of the theory to restrict our attention to a particular class of ``well-behaved'' $\FI$-sets. To this end we have the following definition.\\

\begin{definition}\label{fingen}
An $\FI$--set $Z_\dt$ is said to be \textbf{finitely generated in degree $\leq d$}, if for all $n \geq d$, one has
\[
Z_{n+1} = \bigcup_{f} Z(f)(Z_n)
\]
where the union is over all injections $f:[n] \hookrightarrow [n+1]$.
\end{definition}

The two examples given in Example \ref{firstex} are both finitely generated. The first example is generated in degree 1, while the second is generated in degree $m$.\\

Proofs of the statements in the following theorem can be found in \cite{RW} and \cite{RSW}, although they essentially trace back to the original work of Church, Ellenberg, Farb, and Nagpal in \cite{CEF,CEFN} in the context of $\FI$-modules.\\

\begin{theorem}\label{mainstructurethm}
Let $Z_\dt$ denote an $\FI$-set which is finitely generated in degree $\leq d$. Then,
\begin{enumerate}
\item (The Noetherian Property) every subset of $Z_\dt$ is also finitely generated;

\item (Polynomial stability) for $n \gg 0$, the function
\[
n \mapsto |Z_n|
\]
agrees with a polynomial of degree $\leq d$;

\item (Finiteness of torsion) for $n \gg 0$, and all injections $f:[n] \hookrightarrow [n+1]$, the induced map $Z(f)$ is injective;

\item (Stabilization of orbits) for $n \gg 0$, and all injections $f:[n] \hookrightarrow [n+1]$, the induced map
\[
Z_n / \Sn_n \rightarrow Z_{n+1}/\Sn_{n+1}
\]
is an isomorphism. We denote this limiting orbit set by $Z_\dt / \Sn_\dt$, and call its elements the \textbf{orbits} of $Z_\dt$. \label{stableorbit}
\end{enumerate}
\end{theorem}

\begin{remark}
Throughout this work we will often abuse notation and consider orbits as being subsets of their corresponding $\FI$-set. While this is not literally true, one can justify its usage as follows. Given an orbit $\Ob \in Z_\dt / \Sn_\dt$, one obtains a collection of orbits $\Ob_n \subseteq Z_n$ for $n \gg 0$. We therefore consider the subset of $Z_\dt$ which is the empty set until the maps of (\ref{stableorbit}) begin to be isomorphisms, at which point the subset is defined to agree with $\Ob$. Note that in practice we are usually concerned with asymptotic questions, and therefore it is not particularly important when we decide to have this subset agree with $\Ob$.
\end{remark}

In the paper \cite{RSW}, where $\FI$-sets were first examined, it is argued that many naturally occurring examples of $\FI$-sets come equipped with a collection of $\Sn_n$-equivariant relations. To be more precise, one has the following definition.

\begin{definition}\label{relation}
Let $Z_\dt$ and $Y_\dt$ denote two $\FI$-sets. Then the product $Z_\dt \times Y_\dt$ carries the structure of an $\FI$-set in a natural way. A \textbf{relation} between $Z_\dt$ and $Y_\dt$ is a subset $R_\dt$ of $Z_\dt \times Y_\dt$. If $Z_\dt = Y_\dt$, then we say that $R_\dt$ is a \textbf{relation on $Z_\dt$}

Given a relation $R_\dt$ between $Z_\dt$ and $Y_\dt$ we obtain a family of $\Sn_n$-linear maps
\[
r_n: \RR Z_n \rightarrow \RR Y_n
\]
where $\RR Z_n$ is the $\RR$-linearization of the set $Z_n$, and similarly for $\RR Y_n$. Properties of these maps were a major focus of \cite{RSW}. In this work, they will naturally arise as probability transition matrices of certain families of Markov chains.
\end{definition}

It is a fact, proven in \cite{RSW}, that any relation between two finitely generated $\FI$-sets is itself finitely generated. Perhaps the two most notable classes of examples of $\FI$-set relations arise in the theories of $\FI$-posets and $\FI$-graphs.

\begin{example}\label{figraph}
In \cite{G}, Gadish defines what he calls an $\FI$-poset, a functor from $\FI$ to the category of posets. One may also think of an $\FI$-poset as an $\FI$-set, paired with a relation encoding the ordering. Gadish uses this structure to prove non-trivial facts about linear subspace arrangements.

In \cite{RW}, the authors defined what they called $\FI$-graphs, functors from $\FI$ to the category of graphs and graph homomorphisms. In this case, one may think of an $\FI$-graph as an $\FI$-set of vertices paired with a symmetric relation dictating how these vertices are connected through edges. One should note in this case that the associated linear maps $r_n$ are what one would usually call the \textbf{adjacency matrices} of the corresponding graphs.

Some examples of $\FI$-graphs include the complete graphs $K_n$, whose vertices are the set $[n]$, and whose associated relation is comprised of all pairs $(i,j)$ with $i \neq j$, and the Kneser graphs $KG(n,r)$, whose vertices are given by $r$-element subsets of $n$ and whose associated relation is comprised of all pairs $(A,B)$ such that $A \cap B = \emptyset$. We will see many more examples of $\FI$-graphs throughout the work.
\end{example}

To state the main result of \cite{RSW}, we first need to understand the algebra of relations.

\begin{definition}
Let $Z_\dt$ denote a finitely generated $\FI$-set, and let $R_\dt$ denote a relation on $Z_\dt$. Then $R_n/\Sn_n$ is a subset, in the usual sense, of the set $(Z_n \times Z_n)/\Sn_n$. Thus, at least for $n \gg 0$, one can think of a relation on $Z_\dt$ as some union of orbits of pairs of elements of $Z_n$. If $\Ob \in (Z_\dt \times Z_\dt)/\Sn_\dt$ we will reserve $\{r^\Ob_n:\RR Z_n \rightarrow \RR Z_n\}$ to denote the collection of $\Sn_n$-linear maps induced by the relation given by $\Ob$.
\end{definition}

\begin{proposition}\label{relpoly}
Let $Z_\dt$ be a finitely generated $\FI$--set, and let $\Ob_1,\Ob_2$ denote two orbits of $Z_\dt \times Z_\dt$. Then for each orbit $\Ob$ of $Z_\dt \times Z_\dt$, there exists a polynomial, $p_{\Ob_1,\Ob_2,\Ob}(n) \in \RR[n]$, such that for $n \gg 0$
\[
r_n^{\Ob_1} \circ r_n^{\Ob_2} = \sum_{\Ob} p_{\Ob_1,\Ob_2,\Ob}(n) r^\Ob_n.
\]
\end{proposition}

\begin{proof}
We may think of $r_n^{\Ob_1}$ as a matrix whose rows and columns are labeled by the elements of $Z_n$, and whose entries are given by
\[
(r_n^{\Ob_1})_{x,y} = \begin{cases} 1 &\text{ if $(x,y) \in (\Ob_1)_n$}\\ 0 &\text{ otherwise.}\end{cases}
\]
A similar description exists for $r_n^{\Ob_2}$. It follows that
\[
(r_n^{\Ob_1} \circ r_n^{\Ob_2})_{x,y} = |\{z \mid (x,z) \in (\Ob_1)_n \text{ and } (z,y) \in (\Ob_2)_n\}|.
\]
Therefore, fixing a collection of representatives $(x_n,y_n) \in \Ob_n$,
\[
r_n^{\Ob_1} \circ r_n^{\Ob_2} = \sum_{\Ob} p_{\Ob_1,\Ob_2,\Ob}(n) r^{\Ob}_n,
\]
where,
\[
p_{\Ob_1,\Ob_2,\Ob}(n) = |\{z \mid (x_n,z) \in (\Ob_1)_n \text{ and } (z,y_n) \in (\Ob_2)_n.\}
\]
It remains to show that the function $p_{\Ob_1,\Ob_2,\Ob}(n)$, as defined in the previous line, agrees with a polynomial.

We begin by defining an $\FI$-set
\[
X_n := \{(x,z,y) \mid (x,y) \in \Ob_n, (x,z) \in (\Ob_1)_n, (z,y) \in (\Ob_2)_n\}
\]
whose transition maps are defined in the obvious way. This $\FI$-set is finitely generated, as it is a subset of $Z_\dt \times Z_\dt \times Z_\dt$, and the latter $\FI$-set is finitely generated. Theorem \ref{mainstructurethm} implies that the map
\[
n \mapsto |X_n|
\]
agrees with a polynomial for $n \gg 0$. It follows that
\[
n \mapsto |X_n| / |\Ob_n| = p_{\Ob_1,\Ob_2,\Ob}(n)
\]
agrees with a rational function for $n \gg 0$. On the other hand, it is clear that $p_{\Ob_1,\Ob_2,\Ob}(n)$ is an integer for all $n$. The only rational functions which take integral values on all sufficiently large integers are polynomials. This finishes the proof.
\end{proof}

The content of the above proposition is that it permits us to make the following definition

\begin{definition}\label{virtualrelation}
Let $Z_\dt$ be a finitely generated $\FI$-set, then the \textbf{relation algebra of $Z_\dt$} is defined by the presentation
\[
\R(Z_\dt) := K\langle r^{\Ob}\rangle_{\Ob \in Z_\dt / \Sn_\dt} /(r^{\Ob_1} \cdot r^{\Ob_2} = \sum_{\Ob} p_{\Ob_1,\Ob_2,\Ob}(n) r^\Ob)
\]
where $K$ is the function field $\RR(n)$, and $p_{\Ob_1,\Ob_2,\Ob}$ are the polynomials of Proposition \ref{relpoly}. An element
\[
\sum_{\Ob} a_{\Ob} r^\Ob \in \R(Z_\dt)
\]
is known as a \textbf{virtual relation} of $Z_\dt$. Given a virtual relation $r = \sum_{\Ob} a_{\Ob} r^\Ob \in \R(Z_\dt)$, we can associate a collection of $\Sn_n$-linear maps $r_n: \RR Z_n \rightarrow \RR Z_n$ for all $n \gg 0$ by setting
\[
r_n := \sum_{\Ob} a_{\Ob}(n) r^\Ob_n
\]
We call $r_n$ the \textbf{specialization of $r$ to $n$}.
\end{definition}

\begin{remark}
For the remainder of this paper, we will reserve $K$ to denote the rational function field $K = \RR(n)$.

Also note that the choice of having $K$ be this field is somewhat arbitrary. Indeed, one may replace $\RR$ with any characteristic 0 field. Moreover, as it is sometimes useful to allow for square roots of $n$, one may just as well work with algebraic extensions of $K$. In this paper we have taken $K$ to be $\RR(n)$ simply for ease of exposition.
\end{remark}

The following theorem follows from \cite[Corollary C]{RSW}.

\begin{theorem}\label{eigenvaluestab}
Let $Z_\dt$ denote a finitely generated $\FI$-set, and let $r = \sum_{\Ob} a_{\Ob} r^\Ob \in \R(Z_\dt)$ be a virtual relation on $Z_\dt$. Then for $n \gg 0$
\begin{enumerate}
\item the number of distinct eigenvalues of $r_n$ is independent of $n$;
\item there exists a finite list $\{f_i\}$ of functions which are algebraic over $K$, for which $\{f_i(n)\}$ is the complete list of eigenvalues of $r_n$;
\item for any $f_i$ as in the previous part, the function
\[
n \mapsto \text{ the algebraic multiplicity of $f_i(n)$ as an eigenvalue of $r_n$}
\]
agrees with a polynomial.
\end{enumerate}
\end{theorem}

\begin{example}
Let $Z_\dt$ be the first $\FI$-set discussed in Example \ref{firstex}. Namely, we have
\[
Z_n = [n], Z(f) = f.
\]
Then there are two orbits of $Z_\dt \times Z_\dt$ given by
\[
(\Ob_1)_n := \{(x,y) \mid x \neq y\}, (\Ob_2)_n := \{(x,x) \mid x \in [n]\}.
\]
Writing $r^{(1)} := r^{\Ob_1}, r^{(2)} := r^{\Ob_2}$, we see that every virtual relation on $Z_\dt$ is given by
\[
r = \alpha_1 r^{(1)}+ \alpha_2 r^{(2)}
\]
where $\alpha_i \in K$. For this example, we will examine a few specific choices for the $\alpha_i$.

First, if we set $\alpha_1 = 1$ and $\alpha_2 = 0$, then for $n\geq 0$, $r_n = r^{(1)}_n$ is the adjacency matrix of the complete graph $K_n$. A simple computation verifies that for $n \geq 2$, the eigenvalues of $r_n$ are $n-1$ and $-1$, with multiplicities $1$ and $n-1$, respectively. Next, we may set $\alpha_1 = -1$ and $\alpha_2 = n-1$. In this case $r_n$ is the Laplacian matrix of the complete graph. The eigenvalues of $r_n$ are given by $n$ and $0$, with multiplicities $n-1$ and $1$, respectively. Finally, setting $\alpha_1 = \frac{1}{n-1}$ and $\alpha_2 = 0$, we find that $r_n$ is the transition matrix of the simple random walk on the complete graph.

If $G_\dt$ is a finitely generated $\FI$--graph with vertex $\FI$--set $V(G_\dt)$, it is shown in \cite{RW} that there exist virtual relations $r^A,r^L,$ and $r^P$ on $V(G_\dt)$ such that $r^A_n, r^L_n,$ and $r^P_n$ are the adjacency matrix, the Laplacian matrix, and the transition matrix of a simple random walk on $G_n$, respectively. One of the main results of this paper is that discrete Green's functions and expected hitting times can also be realized as the specialization of some virtual relation. 
\end{example}

With all of the previous background, we are finally ready to define what is meant by a random walk on an $\FI$--set.

\begin{definition}
Let $Z_\dt$ denote a finitely generated $\FI$--set. A virtual relation $r = \sum_{\Ob} a_{\Ob}r^{\Ob} \in \mathcal{R}(Z_\dt)$ is said to be a \textbf{transition relation} if for all $n \gg 0$:
\begin{enumerate}
\item $a_\Ob(n) \in [0,1]$ for all orbits $\Ob$;
\item writing $r_n$ in the standard basis of $\RR Z_n$, the row sums are all 1.
\end{enumerate}
A \textbf{Markov chain on $Z_\dt$} is a sequence of Markov chains $\{X_{t,n}\}_{n \geq 0}$, with state space $Z_n$, such that there exists a transition relation $P$, for which $P_n$ is the transition matrix of $X_{t,n}$ whenever $n \gg 0$.
\end{definition}

\begin{remark}
If $(X_{t,\dt},P)$ is a Markov chain on some finitely generated $\FI$--set $Z_\dt$, then we will often apply terminology from the theory of Markov chains to describe it. When this is done it is meant to communicate that the Markov chains $X_{t,n}$ have a certain property for $n \gg 0$. For instance, a Markov chain $X_{t,\dt}$ on $Z_\dt$ is said to be connected, if $X_{t,n}$ is for all $n \gg 0$. As before, we will often denote a Markov chain by the pair $(X_{t,\dt},P)$
\end{remark}

\begin{example}
Let $G_\dt$ denote a finitely generated $\FI$-graph. For any orbit $\Ob$ of $V(G_\dt)$, and any $n \gg 0$, every vertex $v \in \Ob_n$ has the same degree. Denoting this degree by $\mu(\Ob_n)$, it is proven in \cite{RW} that the function
\[
n \mapsto \mu(\Ob_n)
\]
agrees with a polynomial for all $n \gg 0$. It follows that the probability transition matrix of the simple random walk on $G_n$ agrees with the specialization of a virtual relation on $V(G_\dt)$ to $n$. This virtual relation will be a transition relation by construction. It will also be connected whenever $G_\dt$ is, and reversible
\end{example}

Given a connected Markov chain $(X_{t,\dt},P)$ on a finitely generated $\FI$--set $Z_\dt$, our goal will be to understand various statistics of $X_{t,n}$ as functions of $n$. In particular, we will focus on how hitting times and mixing times vary with $n$.

\section{Markov chains on \texorpdfstring{$\FI$}{FI}-sets: Mixing times}

We would like to understand the mixing times of random walks on \fg s. We will need the classification of \fg s from \cite{RSW}, and this is recalled in Section \ref{sec:finfi}. In Section \ref{sec:mixingex}, we will analyze mixing times of random walks on some explicit examples of \fg s. We will then see in Section \ref{sec:mixingpf} that general reversible random walks on \fg s behave in the same way as these examples.

In Section \ref{sec:mixingpf}, we will consider a particular random walk on an \fg{} --- the weighted walk (Definition \ref{def:weightedwalk}). We will use our understanding of finitely generated \fg s, from \cite{RSW} and Section \ref{sec:finfi}, to understand the behavior of the weighted walk (Theorem \ref{the:orbitwalkmixing}), and then compare other reversible walks on the same \fg{} to this one (Theorem \ref{the:fimixing}). 

The main results are that the weighted walk (which has a specific choice of edge weights, depending on the edge orbits) has mixing time which is constant in $n$, and that any other reversible walk on the same \fg{} has mixing time bounded above by the most extreme ratio of probability flows $\pi(x)P(x,y)$ between this walk and the first. In the language of Definition \ref{graphdef}, this last result applies to any model of a random walk on an \fg.

In particular, if we are interested in the simple random walk on an \fg, then the mixing time is bounded by an expression in terms of degrees of vertices and sizes of edge orbits (Corollary \ref{cor:srw}).

\subsection{Characterization of finitely generated $\FI$--graphs}
\label{sec:finfi}

We will need to work with the vertex orbits and edge orbits of \fg s. Because the vertex set of an \fg{} is an \fs, we may use Theorem A of \cite{RSW}. This theorem says that each (eventual) vertex orbit of an \fg{} $G_\dt$ is determined by a nonnegative integer $k$ and a subgroup $H$ of $S_k$. For sufficiently large $n$, the vertices of $G_n$ in this orbit are ordered $k$--tuples of elements of $[n]$, identified under the action of $H$.

We will often label vertex orbits with colors. The following example is complicated enough to illustrate various types of behavior which may occur.

\begin{Example}\label{ex:variety}
We define the vertex set of an \fg{} $G_\dt$. The degree-$n$ piece $G_n$ has the following vertices
\begin{itemize}
\item A red vertex labeled by each ordered $3$--tuple of elements of $[n]$
\item A blue vertex labeled by each unordered pair of elements of $[n]$
\item A green vertex labeled by each $4$--tuple of elements of $[n]$, with cyclic permutations identified.
\end{itemize}
\end{Example}

\begin{Remark}
We will be concerned with only the long-term behavior of \fg s. There are many ways to modify an \fg{} without changing the eventual behavior. For instance, the \fg{} of Example \ref{ex:variety} could be modified to have green vertices appearing only for $n > 10$, or to have twice as many blue vertices --- corresponding to ordered pairs instead --- for $n < 15$.   
\end{Remark}

Edge orbits of $G_\dt$ are necessarily orbits of pairs of vertices, but not all pairs of vertices are connected by an edge. 

\begin{Example}\label{ex:varietyedges}
We now specify some edges for the \fg{} of Example \ref{ex:variety}. For each choice of distinct integers $a$ through $d$, there are edges connecting
\begin{itemize}
\item the red vertex $(a,b,c)$ and the blue vertex $\{a,b\}$
\item the red vertex $(a,b,c)$ and the blue vertex $\{a,d\}$
\item the red vertex $(a,b,c)$ and the green vertex $(a,b,c,d)$, and
\item the green vertices $(a,b,c,d)$ and $(a,b,d,c)$.  
\end{itemize} 
Remember that in this example, cyclic permutations of the $4$--tuple labeling a green vertex denote the same vertex.
\end{Example}

Note that orbits of pairs of vertices are not the same as pairs of orbits of vertices. The first two types of edge specified in Example \ref{ex:varietyedges} are not in the same edge orbit, because the action of $S_n$ on $G_n$ preserves the number and nature of incidences between tuples labeling vertices.

We will now describe the random walk on the graph of Example \ref{ex:varietyedges}. It will be useful to group together steps along edges in each edge orbit. Temporarily, let us refer to the edges defined by Example \ref{ex:varietyedges} as being of types one through four.

From the red vertex $(a,b,c)$, a step along a random edge of type two results in the blue vertex $\{a,d\}$, with $d$ a random element of $[n]$ different from $a,b,$ and $c$. Likewise, a step along an edge of type three is the same as choosing a random $d$ and moving to the green vertex $(a,b,c,d)$.

From the blue vertex $\{a,b\}$, a step along a random edge of type one is the same as choosing a random $c$ and moving to either $(a,b,c)$ or $(b,a,c)$, and a step along a random edge of type two is choosing random $c$ and $d$ and moving to either $(b,c,d)$ or $(a,c,d)$. As previously, new labels are chosen from those not already occurring.

We will conclude the discussion of Example \ref{ex:varietyedges} by outlining why the mixing time of a specific random walk on this graph does not depend on $n$ for large enough $n$. The general version of this result will be proven in Section \ref{sec:mixingpf}. 

The random walk that we will consider is the following. From any vertex, among adjacent edges choose an orbit of edges uniformly and random, and then choose a random adjacent edge uniformly from that orbit. Move along that edge. We will refer to this walk as the orbit walk, and it will be justified in Section \ref{sec:mixingpf}.

There are two different elements to the mixing of this walk. Firstly, the walk moves between red, blue, and green vertices, and it cannot have mixed before the probability of being at vertices of each color is close to correct. This happens quite quickly, because the sequence of vertex colors is a finite Markov chain with three states, connected state space, a state (green) with nonzero holding probability, and transition probabilities of at least $\frac{1}{3}$. The mixing of this walk does not depend on $n$.

The other element necessary for mixing is that the vertex labels must be close to random, and this takes somewhat longer. A walk started at the red vertex $(a,b,c)$ likely still has $a$ in its vertex label several steps later. Indeed, the only ways to reach states without the label $a$ are that moves from green vertices to red drop a random label, type-two moves from blue to red drop a random label, and moves through blue or green vertices may move the label $a$ into a position other than the first, whence it may be replaced by moving to a blue vertex along an edge of type two. This walk actually mixes faster as $n$ increases, because the time taken to replace all initial labels does not change, while the impact of the new random labels depending on the initial labels decreases. This effect is seen more clearly in Example \ref{ex:refresh}.

Section \ref{sec:mixingpf} will put these ideas together to show, as Theorem \ref{the:orbitwalkmixing}, that the mixing time of the orbit walk on any finitely generated \fg{} is bounded above by a function which does not depend on $n$. We will then consider the impact of orbits of different sizes, resulting in Theorem \ref{the:fimixing}, which bounds the mixing time by the most extreme ratio between orbit sizes.

The following example illustrates how increasing $n$ speeds up the mixing of the vertex labels.

\begin{Example}\label{ex:refresh}
Let $G_n$ have vertices indexed by pairs of elements of $[n]$, with two vertices connected by an edge if the labeling pairs are distinct. 
\end{Example} 

For large $n$, the simple random walk on the graph of Example \ref{ex:refresh} mixes in a single step. After taking a single step from the vertex $\{a,b\}$, the walk is equally likely to be at any of the $\binom{n-2}{2}$ of $\binom{n}{2}$ vertices which do not use either of the labels $a$ or $b$. As $n$ increases, this distribution becomes closer (in total variation distance) to the uniform distribution on all $\binom{n}{2}$ vertices, because proportionally more of the total vertices are included.

\subsection{Mixing times of \texorpdfstring{$\FI$--graphs}{FI--graphs}: Examples}
\label{sec:mixingex}

In this section, we will examine the mixing times of random walks on some examples of $\FI$--graphs. We will see in Section \ref{sec:mixingpf} that reversible random walks on general finitely generated \fg s behave in the same way.

Our simplest examples of \fg s had few vertex orbits and many edges. 

\begin{Example}\label{ex:completemixing}
The simple random walk on the complete graph $K_n$ mixes in a single step.
\end{Example}
\begin{proof}
Taking a single step from any vertex results in a uniform distribution on the other $n-1$ vertices. This is a distance of $\frac{1}{n}$ from uniform.
\end{proof}

As noted in Remark \ref{rem:mixingeps}, the random walk of Example \ref{ex:completemixing} mixes in a single step in the sense that for any $\eps$ there is an $N$ so that for all $n > N$, one step of the random walk on $K_n$ is enough to get within $\eps$ of uniform. 

\begin{Example}\label{ex:knesermixing}
The Kneser graph $KG(n,r)$ has vertices labeled by subsets of $[n]$ of size $r$, with two vertices being joined by an edge if they are labeled by disjoint $r$--tuples. For fixed $r$ and large enough $n$, the simple random walk on this graph mixes in a single step, in the sense that for any $\epsilon$ it is possible to choose a large enough $n$ so that after one step the distribution is within $\epsilon$ of uniform in total variation distance. 
\end{Example}
\begin{proof}
Consider a walk started at a given $r$--tuple $\mathbf{x}$. After one step, the walk is uniformly distributed over $r$--tuples which are disjoint from $\mathbf{x}$. But as $n$ grows, the proportion of $r$--tuples which are disjoint from $\mathbf{x}$ approaches $1$, so the total variation distance of this distribution from uniform approaches $0$.
\end{proof}

\begin{Example}\label{ex:johnsonmixing}
The Johnson graph $J(n,r)$ has vertices labeled by subsets of $[n]$ of size $r$, with two vertices being joined by an edge if the corresponding subsets have $r-1$ elements in common. For fixed $r$ and large enough $n$, the simple random walk on this graph mixes in approximately $r\ln(r)$ steps. More precisely, for each $c$ we have that $$\tm(e^{-c}) \leq r\ln(r) + cr.$$
\end{Example}
\begin{proof}
Note that a step of the simple random walk on this graph may be seen as choosing a random element of the present subset $A$, removing it, and replacing it by any other element of $[n] - A$.

We prove the result by coupling, which gives bounds on mixing times by Theorem \ref{the:couplingmixing}. Consider two separate instances of the random walk. We couple the walks so that the the two states --- which are $r$--element subsets of $[n]$ --- tend to have more and more elements in common. To construct the coupling, we consider each possible step the first chain might take, and choose a corresponding step for the second. Let $A_t$ and $B_t$ be the states at time $t$.

The first part of taking a step in either chain is to choose an element to remove from $A_t$. If an element is chosen which is also in $B_t$, then remove it from $B_t$ as well. Otherwise, remove an arbitrary element.

Once an element has been removed from each of the chains, we need to choose replacement elements. Let $f$ be a bijection from $[n] - A_t$ to $[n] - B_t$ with $f(x) = x$ for each $x$ which is in both $[n] - A_t$ and $[n] - B_t$. When the element $x$ is chosen as the replacement in the first chain, choose $f(x)$ in the second. 

We are interested in the time taken until this coupling results in both chains being in the same state. Let $k$ be the number of elements in common on the sets $A_t$ and $B_t$. For simplicity, we will refer to these common elements as `matches'. Then:
\begin{itemize}
\item Removing elements either removes a match, with probability $\frac{k}{r}$, or does not change the number of matches.
\item Replacing elements either adds one match, with probability $\frac{n - (2r-k)}{n-r}$, or adds two matches, with probability $\frac{r-k}{n-r}$.
\end{itemize}

Together, these steps never decrease the number of matches, and strictly increase the number of matches with probability at least $\frac{r-k}{r} + \frac{r-k}{n-r}$. This probability is at least $\frac{r-k}{r}$, so the time until there are $r$ matches is less than the coupon collector time $T$ for $r$ coupons.

It is a standard result (see for instance Proposition 2.4 of \cite{LPW}), that $$\mathbb{P}(T > \lceil r\ln(r) + cr\rceil) \leq e^{-c},$$ so $$\tm(e^{-c}) \leq r\ln(r) + cr.$$
\end{proof}

In Example \ref{ex:johnsonmixing}, it is important that the size of the $r$--tuples is fixed. If larger graphs had vertices indexed by larger tuples, then the mixing time might no longer be constant. However, a finitely generated \fg{} cannot exhibit this behavior, by Theorem $A$ of \cite{RSW} applied to the vertex set.

Future examples will often reduce to the same idea as Example \ref{ex:johnsonmixing} --- that the mixing time is essentially the time required to refresh each element of the subset or tuple labelling the vertices. We shall not repeat these calculations in so much detail. According to Theorem A of \cite{RSW}, each \fg{} that we consider will have vertices described by tuples, potentially identified by the action of a permutation group.  

Examples \ref{ex:knesermixing} and \ref{ex:johnsonmixing} could have been defined using ordered $r$--tuples rather than unordered without any change in the results. 

Examples \ref{ex:completemixing}, \ref{ex:knesermixing}, and \ref{ex:johnsonmixing} have mixing times which are bounded above by functions constant in $n$. (Indeed, our bounds in each case actually decrease slightly with $n$, as we rely on the largest orbits of vetices or edges, and these grow comparatively larger with $n$). Our next examples are of \fg s with slower mixing.

\begin{Example}\label{ex:differentorbits}
Let each graph $G_n$ have vertices labeled by $4$--tuples of elements of $[n]$. There is an edge between two vertices if their $4$-tuples either agree in their first coordinate or agree in their last three coordinates. Then the simple random walk on $G_n$ has mixing time $O(n^2)$.
\end{Example}
\begin{proof}
For the upper bound, we couple two instances of the random walk so that when either replaces its first coordinate, they both do, choosing the same new value if possible. This occurs with probability at least $\frac{n-8}{(n-4) + \binom{n-4}{3}}$. After this happens, the next time the last three coordinates are replaced, we choose the same values if possible. The probability of coupling in one step, given that the first coordinates already match, is at least $\frac{\binom{n-7}{3}}{\frac{n-8}{(n-4) + \binom{n-4}{3}}}$.

For large $n$, it takes $O(n^2)$ steps to match the first coordinates, and then only one more step to match the rest, so the mixing time is at most $O(n^2)$. 

For the lower bound, notice that only $\frac{1}{n}$ of the vertices of $G_n$ have the same first coordinate as the starting vertex, but it likely takes $O(n^2)$ steps for the first coordinate to change.
\end{proof}

Modifying the edge orbits of Example \ref{ex:differentorbits} can produce graphs whose mixing time is $O(n^r)$, for any $r$, by using $(r+2)$--tuples for vertices, and having edges between vertices which differ only in the first coordinate or which agree in the first coordinate. The limiting factor in the mixing time calculation is the time until the unlikely step of replacing the first coordinate. Similar behavior occurs in the next example.

\begin{Example}\label{ex:bottleneck}
Let each graph $G_n$ have $\binom{n}{r}$ red vertices and $\binom{n}{r}$ blue vertices, each labeled by subsets of $r$ elements of $[n]$, and a single green vertex. There are edges between each pair of red vertices, each pair of blue vertices, and between the green vertex and any other vertex. Then the simple random walk on $G_n$ has mixing time $O(n^r)$.
\end{Example}
\begin{proof}
For the upper bound, we couple two instances of the walk. If neither is at the green vertex, then couple them so that if either moves to the green vertex, they both do. Otherwise, wait a step and try again. It takes $O(n^r)$ steps for the two walks to couple.

For the lower bound, a walk starting at a red vertex takes $O(n^r)$ steps until it reaches a non-red vertex, but the stationary measure of the set of red vertices is less than $\frac{1}{2}$.
\end{proof}

In both Example \ref{ex:differentorbits} and Example \ref{ex:bottleneck}, the mixing time is slow because the walk cannot have mixed until it has taken a step along an edge belonging to an unlikely orbit. We will now see that adjusting for the relative likelihood of the various edge orbits produces mixing times which do not grow with $n$, as in Examples \ref{ex:completemixing} to \ref{ex:johnsonmixing}.

\begin{Example}\label{ex:differentorbitsweighted}
Modify Example \ref{ex:differentorbits} by changing the transition probabilities. The new transition rule is to replace the first coordinate with probability $\frac{1}{2}$ and to replace the last three coordinates with probability $\frac{1}{2}$, choosing uniformly from edges in those orbits in each case.

This modified walk is close to mixed as soon as each type of edge has been chosen at least once, which takes on average two steps.
\end{Example}

\begin{Example}\label{ex:bottleneckweighted}
Modify Example \ref{ex:bottleneck} as follows
\begin{itemize}
\item From a red vertex, move to either a random red vertex or to the green vertex, each with probability $\frac{1}{2}$. 
\item From a blue vertex, move to either a random blue vertex or to the green vertex, each with probability $\frac{1}{2}$.
\item From the green vertex, move to a random other vertex.
\end{itemize}
As with Example \ref{ex:differentorbitsweighted}, this walk has been modified to choose edge orbits uniformly, rather than with probability proportional to their size. Unlike that example, in this case the modification has changed the stationary distribution, drastically increasing the probability of being at the green vertex. 

This modified walk is close to mixed as soon as it has taken a step from the green vertex, and this takes on average three steps. 
\end{Example}

\subsection{Mixing times of \texorpdfstring{$\FI$--graphs}{FI--graphs}: General case}
\label{sec:mixingpf}

In this section, we will show that general finitely generated \fg s behave in the same way as our examples in Section \ref{sec:mixingex} --- that is, that they may have large mixing times, but only due to the relative probabilities of steps along various orbits of edges. Adjusting for these frequencies, as in Examples \ref{ex:differentorbitsweighted} and \ref{ex:bottleneckweighted}, gives mixing times which do not grow with $n$. The following definition generalizes those examples.

\begin{Definition}\label{def:weightedwalk}
Let $G_\dt$ be an \fg. For each $n$, the weighted walk on $G_n$ is defined as follows. If the current vertex is $v$, then consider all edges containing $v$. Choose one of those edge orbits uniformly at random, and then one of those edges uniformly from that orbit.
\end{Definition}

We will also want to consider which vertex orbit our random walk is at.

\begin{Definition}\label{def:orbitgraph}
For either the simple random walk or the weighted walk on an \fg, the vertex orbits form a quotient Markov chain. This is the \textbf{orbit walk}, or a random walk on the \textbf{orbit graph}, whose vertices and edges are vertex orbits and edge orbits. That is, the information of which vertex orbit the walk is at at time $t$ is sufficient to determine the probability distribution on vertex orbits at time $t+1$.
\end{Definition}

It will be convenient to consider a more concrete modification of our \fg s. As discussed at the start of Section \ref{sec:finfi}, Theorem A of \cite{RSW} says that the vertices of a finitely generated \fg{} are ordered tuples, identified under the action of a permutation group. If we just neglect to identify any such vertices, then we have a larger graph whose vertices are ordered tuples, with any two connected by an edge whenever their images in the actual \fg{} are. The random walk on the actual \fg{} is a quotient Markov chain of the random walk on the larger graph, so any upper bound on the mixing time of the larger random walk also bounds the mixing time for the random walk on our actual \fg. When we study simple random walks on the smaller graph, these are not necessarily the image of simple random walks on the large graph, so we may still need to work with unequal transition probabilities.

In the following analysis we shall derive upper bounds using this simplification, working with \fg s whose vertices are just ordered tuples.

We will also need to track particular edge labels.

\begin{Definition}\label{def:augorbitgraph}
For any $k \in [n]$, the \textbf{augmented orbit graph} is defined in the same way as the orbit graph (Definition \ref{def:orbitgraph}), except that instead of vertices being vertex orbits, they are vertex orbits together with either the position of $k$ within the label of the present vertex, or the information that $k$ does not occur in this label.
\end{Definition}

\begin{Theorem}\label{the:orbitwalkmixing}
For any finitely generated \fg{} which is eventually not bipartite, the weighted walk on $G_n$ has mixing time which is constant in $n$, in the sense of Remark \ref{rem:mixingeps}.
\end{Theorem}
\begin{proof}
Essentially, this result is because the behavior of the weighted walk does not depend on $n$, so the time taken for events like ``the walk has replaced all of its starting labels and is now very close to random'' also does not depend on $n$.

The condition that the \fg{} be eventually not bipartite is to guarantee that the weighted walk is aperiodic. 

As discussed in Section \ref{sec:finfi}, to show that the weighted walk has mixed we need to deal with mixing on the set of vertex orbits and with mixing of the vertex labels.

Let $N$ be large enough that $G_N$ contains representatives of each edge orbit and vertex orbit of any $G_n$. Take any $n$ with $n \geq 2N$ and consider the weighted walk on $G_n$.  

We will couple two copies of this walk. 

Firstly, consider the two walks on the orbit graph, ignoring the vertex labels. This is a finite graph which does not depend on $n$ (once $n \geq N$), so there is a coupling whose coupling time does not depend on $n$ (It is important here that we are considering the weighted walk, so the transition probabilities do not depend on $n$. If we were working with the simple random walk on the original \fg{}, then the transition probabilities for the induced walk on the orbit graph could depend on $n$. This coupling also requires that this walk on the orbit graph is aperiodic, which follows from the \fg in question not being bipartite). Couple the two walks together according to this coupling until they are at the same vertex orbit. Let the time at which this happens be $T$. After this time, couple the two chains so that they move along the same edge orbits, with any new labels to be determined momentarily.

It remains to couple two chains which start at vertices in the same orbit, $V$, but with different labels. Let these vertices be $u_0$ and $v_0$, and $f_0$ a bijection from $[n]$ to $[n]$ which takes $u_0$ to $u_1$ when applied to the labels. Among such bijections $f_0$, choose one which fixes as many elements of $[n]$ as possible --- this will be at least $n - 2N$. Essentially, $f_0$ pairs labels of $u_0$ with similarly-positioned labels of $v_0$, and vice versa.

As our two chains $u_t$ and $v_t$ progress, we will couple them so that the function $f_t$ moves closer and closer to being the identity. At time $t$, the first chain moves from $u_{t-1}$ to $u_t$, and the second moves from $v_{t-1} = f_{t-1}(u_{t-1})$ to $v_t := f_{t-1}(u_t)$. Next, we define the function $f_t$. We could define $f_t$ to be equal to $f_{t-1}$, but we may be able to achieve $v_t = f_t(u_t)$ with a choice of $f_t$ which is closer to the identity --- as measured by its number of fixed points. As with the initial choice of $f_0$, just choose $f_t$ arbitrarily from among bijections on $[n]$ which take $u_t$ to $v_t$ and have the most possible fixed points. This will produce an $f_t$ with at least as many fixed points as $f_{t-1}$, because $f_{t-1}$ was one possible choice for $f_t$.

Once $f_t$ is the identity, then $u_t = v_t$ and the two chains have coupled. Observe that if $f_t$ fixes the label $i$, then all later functions $f_{t + t'}$ also fix $i$. Thus it suffices to wait until each label has been fixed by some $f_t$. Also, if the label $i$ does not label either $u_t$ or $v_t$, then $f_t$ fixes $i$, so the coupling time is bounded above by the time taken for each element of $[n]$ labeling $u_0$ (equivalently, $v_0$) to at some point be removed from the sets of labels of $u_t$ and $v_t$ (even if it later returns, because after such a label returns it will forevermore be in the same position labeling $u_t$ and $v_t$). 

All that remains is to, for each label $i$, bound the time until that label is removed. If each of these times have bounds which do not depend on $n$, then the time until this has happened for each of up to $2N$ labels has bounds which do not depend on $n$, which completes the proof. We shall give appropriate bounds in a moment as Proposition \ref{prop:constremoval}. 
\end{proof}

\begin{Proposition}
\label{prop:constremoval}
Let $\eps$ be any small constant and $k \in [n]$ be any label. Then for the coupling considered in the proof of Theorem \ref{the:orbitwalkmixing}, there is a constant $T_\eps$ so that with probability at least $1-\eps$ there is some $t < T_\eps$ so that at time $t$, neither chain has $k$ labeling its present state. In the notation of the previous proof, $k$ is not a label of $u_t$ or $v_t$. The constant $T_\eps$ does not depend on $n$. 
\end{Proposition}
\begin{proof}
For this proof, we will work with the augmented orbit graph of Definition \ref{def:augorbitgraph} rather than the orbit graph. We may see this as just keeping track of an additional piece of information --- the position of the label $k$, if any --- as we run the random walk. Note that there is a natural projection from the augmented orbit graph to the orbit graph, and that each vertex of the latter has only finitely many preimages. 

From any vertex of the augmented orbit graph (a vertex orbit and a choice of position for $k$), it is possible to get to a state where the next move may remove $k$ from the labels of $v_t$, and then to make that move. Because the (augmented) orbit graph is finite and connected, this becomes arbitrarily likely as we take more steps. After this, it is possible to do the same to remove $k$ from the labels of $u_t$. When this is accomplished, there is at most a probability of $\frac{N}{n} \leq \frac{1}{2}$ that $k$ has reappeared as a label of $v_t$, because each possible label is equally likely to be added, and $v_t$ has at most $N$ labels out of $n$ possible. 

If $k$ is now labeling neither $u_t$ nor $v_t$, then we are done. Otherwise, we proceed until $k$ is not a label of $v_t$, alternating these steps as necessary. The number of iterations required is at most a geometric random variable with parameter $\frac{1}{2}$. 

The claim follows by putting together these estimates. For any $\epsilon$, choose $m$ so that $$\Pr(\Geom(\frac{1}{2}) < m) \geq 1 - \frac{\eps}{2},$$ and then choose $T_\eps$ large enough that for any possible starting locations of the label $k$ in $u_0$ and/or $v_0$, the time until $k$ does not label $v_t$ is less than $\frac{T_\eps}{m}$ with probability at least $1 - \frac{\eps}{2m}$. The probability that $k$ still labels either $u_t$ or $v_t$ at time $T_\eps$ is at most $\frac{\eps}{2} + m\frac{\eps}{2m}$, completing the proof. 
\end{proof}

Theorem \ref{the:orbitwalkmixing} shows that the mixing times of Examples \ref{ex:differentorbits} and \ref{ex:bottleneck} are due not to the structure of the vertex and edge orbits, but rather to the relative sizes of the various edge orbits, which make some orbits much more or less likely than others. We will now bound the mixing time of general reversible random walks (models of random walks on our \fg) by comparing their transition probabilities to those of the weighted walk. In the case of the simple random walk, this depends on how much the orbit sizes differ from one another.

Firstly, let us describe a construction of a general reversible random walk on a finite undirected graph $G$, as in Chapter 2 of \cite{LP}, for instance. Assign a weight $w_e$ to each edge $e$. To take a step from any vertex $v$, choose a neighboring edge at random, each edge chosen with probability proportional to its weight, and move along that edge. Let $\gamma$ be the sum of $w_e$ over all edges, with edges between distinct vertices counted twice. For each vertex $v$, the stationary probability $\pi(v)$ is equal to the sum of $w_e$ over neighbors $e$ of $v$, divided by $\gamma$.

We will need a measure of the greatest difference between transition probabilities in a random walk and the corresponding probabilities in the weighted walk.

\begin{Definition}
Let $G_\dt$ be a finitely generated \fg{}, and consider any model $(X_t,P)$ of a random walk on $G_\dt$. Let $P$, $\pi$, $\gamma$, and $w_e$ refer to the respective quantities for the random walk $P$, and $P'$, $\pi'$, $\gamma'$, and $w'_e$ refer to those quantities for the weighted walk.
\end{Definition}

\begin{Definition}\label{def:rho}
For any edge $e$ of $G_n$, consider the ratio $\frac{w_e / \gamma}{w'_e / \gamma'}$. Define the ratio $\rho(n)$ to be the smallest value of this ratio over all choices of $e$.
\end{Definition}

Notice that the definition of $\rho(n)$ compares corresponding transition probabilities, but also depends on the weights assigned elsewhere. We could remove the dependence on $\gamma$ by requiring that the weights sum to $1$ over the entire graph, so $\gamma = \gamma' = 1$, but this restriction makes it more difficult to imagine the effect of a local change.

In the case of the simple random walk, this ratio may be expressed in terms of vertex degrees and edge orbits.

\begin{Remark}\label{rem:simpleex}
Let $G_\dt$ be a finitely generated \fg{}. For the simple random walk on $G_n$, let $N$ the number of pairs of a vertex $v$ and an edge orbit containing $v$, $\Ob_e$ the set of edges containing $v$ from any one of these edge orbits, and $E$ be the set of edges. The ratio $\rho(n)$ for the simple random walk is the smallest value of the ratio $\frac{|\Ob_e|\times N}{2|E|}$ over all choices of $v$ and $\Ob_e$.
\end{Remark}

Notice that when $\rho(n)$ is equal to $1$, the transition probabilities are equal to those of the weighted walk, so the walk in question is the weighted walk. The ratio $\rho(n)$ is always at most $1$, and smaller values indicate the existence of transitions which are much less likely than they would be in the weighted walk. 

In the case of the simple random walk, $\rho(n)$ being less than $1$ indicates the existence of at least one edge orbit which has comparatively few edges, and thus is less likely to be chosen by the simple random walk compared to the weighted walk.  

It also happens that $\rho(n)$ is eventually equal to a rational function in $n$, by results in Section 4 of \cite{RW}. A good example for the following results is Example \ref{ex:differentorbits} compared to Example \ref{ex:differentorbitsweighted} --- in Example \ref{ex:differentorbits}, the smaller edge orbit has only $\frac{1}{n^2}$ as many edges as the other. Thus $\rho(n)$ is proportional to $\frac{1}{n^2}$, and the following results will show that this means that Example \ref{ex:differentorbits} has a mixing time at most proportional to $n^2$, agreeing with our previous calculations.

\begin{Theorem}\label{the:fimixing}
If $G_\dt$ is a finitely generated \fg{} which is eventually not bipartite, then a general model of a random walk on $G_n$ has mixing time bounded above by a constant times $\rho(n)^{-1}$. 
\end{Theorem}
\begin{proof}
The proof is the same as that of Theorem \ref{the:orbitwalkmixing}, but with weights on the edges of the orbit graph, which are not equal to one another and which may depend on $n$. The only change required to that proof is in the estimation of how long it will take until each label has at some point been unused. We will prove the appropriate estimate --- an analogue of Proposition \ref{prop:constremoval} --- as Proposition \ref{prop:constremoval2}. 
\end{proof}

The following result is the same as that of Proposition \ref{prop:constremoval}, except that the bound is a constant multiple of $\rho(n)^{-1}$, rather than a constant.

\begin{Proposition}
\label{prop:constremoval2}
When the proof of Theorem \ref{the:orbitwalkmixing} is applied to the setting of Theorem \ref{the:fimixing}, then we have the following estimates.

Let $\eps$ be any small constant and $k \in [n]$ be any label. Then there is a constant $T_\eps$ so that with probability at least $1-\eps$ there is some $t < T_\eps \rho(n)^{-1}$ so that at time $t$, neither chain has $k$ labeling its present state. In the notation of the previous proof, $k$ is not a label of $u_t$ or $v_t$. The constant $T_\eps$ does not depend on $n$. 
\end{Proposition}
\begin{proof}
The proof is the same as that of Proposition \ref{prop:constremoval}, except that the time until there is a certain probability that the label $k$ has been removed from either one of $u_t$ or $v_t$ will be (at most) proportional to $\rho(n)^{-1}$, rather than being constant in $n$. This is because, compared to the random walk of Theorem \ref{the:orbitwalkmixing}, steps may have been made less likely by a factor of up to $\rho(n)$. 

Bounds of the desired form follow from, for example, the Commute Time Identity (see Corollary 2.21 of \cite{LP}) and Markov's inequality (to use the commute time to bound the time taken until a certain edge is traversed, subdivide that edge). Recall that we are working in the augmented orbit graph.

In more detail, let $x$ be the starting state. Choose an edge corresponding to the removal of the label $k$, and subdivide it, producing a new vertex $y$. Let $\mathcal{C}(x,y)$ be the conductance between $x$ and $y$. The Commute Time Identity gives that the expected time for the label $k$ to be removed is at most $\frac{\gamma}{\mathcal{C}(x,y)}$. But we know that $\frac{\gamma'}{\mathcal{C}'(x,y)}$ does not depend on $n$, because this is a commute time for a random walk which does not depend on $n$. Moving from the conductance $\mathcal{C}'$ to $\mathcal{C}$, each individual conductance has decreased by at most a factor of $\rho(n)^{-1}\frac{\gamma'}{\gamma}$. This means that $$\mathcal{C} \geq \mathcal{C'}\rho(n)^{-1}\frac{\gamma'}{\gamma}.$$ Therefore this expected commute time between $x$ and $y$ on the (augmented) orbit graph for the random walk $P$ is at most $\rho(n)^{-1}$ times the corresponding expected commute time for the weighted walk, which did not depend on $n$. Markov's inequality completes the proof.
\end{proof}

\begin{Corollary}\label{cor:srw}
If $G_\dt$ is a finitely generated \fg{} which is eventually not bipartite, then the simple random walk on $G_n$ has mixing time bounded above by a constant times $\rho(n)^{-1}$, where $\rho(n)$ is given in Remark \ref{rem:simpleex}.
\end{Corollary}

Theorem \ref{the:fimixing} applies to random walks on \fg s with edges which may have arbitrary weights, which may even depend on $n$, but which must be undirected. If directed edges are allowed, then the random walk may have much worse mixing times. The following example of a directed random walk illustrates this behavior.

\begin{Example}
Let $G_n$ have one red vertex, one orange vertex, one yellow vertex, and a large number of green vertices. Let $p$ be a probability, and consider the following random walk.
\begin{itemize}
\item From the red vertex, move to the orange vertex with probability $p$, else stay put.
\item From the orange vertex, move to the yellow vertex with probability $p$, else move to the red vertex.
\item From the yellow vertex, move to a random green vertex with probability $p$, else move to the red vertex.
\item (Movement from the green vertex is unimportant for this example)
\end{itemize}
When $p=\frac{1}{2}$, the mixing time is a constant. If we change $p$ to $\frac{1}{n}$, then we might hope that with no (directed) edge orbit having changed in probability by a factor of more than $n$, the mixing time might be linear in $n$, as was the case when we moved from Theorem \ref{the:orbitwalkmixing} to Theorem \ref{the:fimixing}. The mixing time is actually cubic, because it takes $O(n^3)$ steps to get to any of the green vertices. 
\end{Example}

This example shows why Theorem \ref{the:fimixing} requires that the random walk be reversible.  

A natural next question about the mixing of (simple) random walks on \fg s is whether or not they exhibit cutoffs. It is shown in \cite{PS} that the Kneser graphs $KG(2n+k,n)$ have cutoff when $k = O(n)$. Theorem A of \cite{RSW} may be understood as saying that arbitrary finitely generated \fg s are essentially built from slightly-generalized Kneser graphs, so it is reasonable to ask whether arbitrary finitely generated \fg s have cutoffs. The answer to this question is no, as illustrated by the following example.

\begin{Example}\label{nocutoff}
Let $G_n$ have $n$ red vertices and $n$ blue vertices, each indexed by the elements of $[n]$. Each vertex is connected to each other vertex of the same color, and also to the oppositely-colored vertex with the same label.

After one step, the walk is within $\frac{1}{2}$ of uniform, but the mixing time is at least linear in $n$, because the probability of moving to a vertex of a different color is only $\frac{1}{n}$. Therefore this random walk does not have cutoff.
\end{Example}

It would be interesting to know whether there are conditions on finitely generated \fg s which guarantee that the simple random walk will have cutoff.

\section{Markov chains on \texorpdfstring{$\FI$}{FI}-sets: Hitting times}

\subsection{Roofed orbit walks}

One tool, which will be invaluable in what follows, is to associate to a connected Markov chain on $Z_\dt$ a new chain on orbits as discussed in the following definition.

\begin{definition}\label{orbitgraph}
Let $Z_\dt$ be a finitely generated $\FI$--set, and let $(X_{t,\dt},P)$ be a connected Markov chain on $Z_\dt$. Viewing $P_n$ as a matrix on $\RR Z_n$, we observe that for any $x,y \in Z_n$ and any $\sigma \in \Sn_n$, one has
\[
P_n(x,y) = P_n(\sigma x, \sigma y),
\]
from the fact that $P_n$ is the specialization of a virtual relation on $Z_\dt$. Therefore the value of $P_n$ is only dependent on the stable orbit of $Z_\dt \times Z_\dt$ that the pair $(x,y)$ belongs to. If $\Ob$ is a stable orbit of $Z_\dt \times Z_\dt$, we will often write $P_n(\Ob_n)$ to denote $P_n(x,y)$ where $(x,y) \in \Ob_n$.

Fix any $m \gg 0$, as well as some $x \in Z_m$. Then we define a family of directed graphs $\{G^x_n\}_{n \geq m}$, called the \textbf{$x$-roofed orbit graph} in the following way:
\begin{itemize}
\item The vertices of $G^x_n$ are labeled by orbits of $Z_n \times Z_n$ which have a representative of the form $(y,x(n))$, where $y \in Z_n$ and $x(n) = Z(\iota_{m,n})(x)$ for $\iota_{m,n}:[m] \hookrightarrow [n]$ the standard injection
\item The vertices $\Ob_n$ and $\Ob_n'$ are connected if there exists $y,z \in Z_n$ such that $(y,x(n)) \in \Ob_n$, $(z,x(n)) \in \Ob_n'$, and $P_n(y,z) > 0$.
\end{itemize}
As a convenient shorthand, we will often write $[z,x(n)]$ to denote the orbit $\Ob_n$ which $(z,x(n))$ is a member of. If a specialization of a stable orbit appears as a vertex of $G_n^x$, then we say that the stable orbit is \textbf{roofed at $x$}, or \textbf{$x$-roofed}.

For each $n \geq m$ we can define a connected Markov chain on $G^x_n$, $(X^x_{t,n},P_n^x)$, by setting
\[
P^x_{n}([z,x(n)],[y,x(n)]) := \sum_{(w,x(n)) \in [y,x(n)]} P_n([z,w])
\]
The fact that this is a connected Markov chain follows from the fact that $P_n$ was the transition matrix of a connected Markov chain on $Z_n$.
\end{definition}

Given a connected Markov chain on a finitely generated $\FI$--set, the defined Markov chains on the orbit graphs can be thought of as encoding the probability that one moves from a given roofed orbit to another one while performing the original Markov chain. This is illustrated with the following example.

\begin{example}
Let $Z_\dt$ be the $\FI$--set with $Z_n = [n]$, and let $(X_{t,\dt},P)$ be the Markov chain on the complete graphs encoding the simple random walk. That is, for each $n$, and each $i,j \in [n]$,
\[
P_n(i,j) = \begin{cases} \frac{1}{n-1} &\text{ if $i \neq j$}\\ 0 &\text{ otherwise.}\end{cases}
\]
In this case there is are two stable orbits of $Z_\dt \times Z_\dt$, pairs of non-equal points and pairs of equal points. We denote these stable orbits by $\Ob^{\neq}$ and $\Ob^{=}$, respectively. Then for $n \geq 3$ the graph $G^1_n$ has two vertices, labeled by $\Ob^{\neq}_n$ and $\Ob^{=}_n$, an edge connecting them, and a loop on the vertex $\Ob^{\neq}_n$. Note that in this case the action of $\Sn_n$ on $Z_n$ is transitive, and so this is the only roofed orbit graph. We also have
\begin{eqnarray*}
P_n^1(\Ob^{=}_n,\Ob^{=}_n) = 0\\
P_n^1(\Ob^{\neq}_n,\Ob^{=}_n) = \frac{1}{n-1}\\
P_n^1(\Ob^{=}_n,\Ob^{\neq}_n) = 1\\
P_n^1(\Ob^{\neq}_n,\Ob^{\neq}_n) = \frac{n-2}{n-1}
\end{eqnarray*}
Thus, the entries of $P_n^1$ are rational functions in $n$. Our first main lemma will show that this is always the case.
\end{example}

\begin{lemma}\label{roofedorbit}
Let $Z_\dt$ be a finitely generated $\FI$--set with a connected Markov chain $(X_{t,\dt},P)$, and let $G_n^x$ and $P_n^x$ be as in Definition \ref{orbitgraph}. Then for any two stable orbits $\Ob,\Ob'$ of $Z_{\dt} \times Z_{\dt}$ which are roofed at $x$, the function
\[
n \mapsto P_n^x(\Ob_n,\Ob'_n)
\]
agrees with a rational function for all $n \gg 0$.
\end{lemma}

\begin{proof}
Write $\Ob_n = [z,x(n)]$ for some $z \in Z_n$ Then we have,
\[
P_n^x(\Ob_n,\Ob'_n) = \sum_{(y,x(n)) \in \Ob'_n} P_n([z,y])
\]
Gathering those $y'$ for which $(y',x(n)) \in \Ob'_n$ and $[z,y] = [z,y']$, the above sum can be written
\[
\sum_{\Ob'' \text{ roofed at } x} |\{y' \mid (y',x(n)) \in \Ob'_n, [z,y'] = \Ob_n''\}| P_n(\Ob_n'')
\]
By assumption, $n \mapsto P_n(\Ob_n'')$ agrees with a rational function, as $P$ is a transition relation. We claim that $n \mapsto |\{y' \mid (y',x(n)) \in \Ob'_n, [z,y'] = \Ob''\}|$ agrees with a polynomial. Indeed, the proof of this is similar to that given in the proof of Proposition \ref{relpoly}.
\end{proof}

As a consequence of this lemma, we see that $P$ can be viewed as a single matrix, as opposed to a collection thereof, over a $K$ vector space. In particular, the construction of the roofed orbit graphs allows us to reduce questions about the infinitely many Markov chains $\{X_{t,n}\}_{n \geq 0}$ to linear algebraic questions of a matrix. This general philosophy, to reduce the behaviors of an infinite collection of objects to the behavior of a single object is woven through the study of $\FI$-sets, and $\FI$-modules.

\begin{definition}
Let $Z_\dt$ be a finitely generated $\FI$--set, and let $(X_{t,\dt},P)$ be a connected Markov chain on $Z_\dt$. Fix some $m \gg 0$, and an element $x \in Z_m$. We say that a stable orbit $\Ob$ of $Z_\dt \times Z_\dt$ is \textbf{$x$-roofed} if for some - and therefore all - $n \geq m$ there is some $y \in Z_n$ such that $(y,x(n)) \in \Ob_n$, where $x(n) = Z(\iota_{m,n})(x)$.

Using the previous paragraph's notation, we define a $K$ vector space $V^{x}_{Z_\dt}$ to be the $K$-linearization of the set of $x$-roofed orbits. Lemma \ref{roofedorbit} implies that the transition matrix $P$ induces a linear endomorphism of $V^{x}_{Z_\dt}$, which we denote by $P^x$.
\end{definition}

\subsection{Hitting times of Markov chains on $\FI$-sets}\label{hittingtime}

In this section we concern ourselves with how the hitting times of a connected Markov chain $X_{t,\dt}$ on a finitely generated $\FI$--set $Z_\dt$ vary with $n$. Our main result will state that these quantities eventually agree with rational functions. The key technique we will use to accomplish this is reducing the problem to a finite problem by using roofed orbit graphs (Definition \ref{orbitgraph}).

\begin{definition}
Let $Z_\dt$ denote a finitely generated $\FI$--set, and let $(X_{t,\dt},P)$ denote a Markov chain on $Z_\dt$. Then we will write $Q_n:\RR Z_n \rightarrow \RR Z_n$ to denote the hitting time matrix with entries
\[
Q_n(x,y) = \mathbb{E}(\tau_{y,n} \mid X_{0,n} = x)
\]
where $\tau_{y,n}$ is the hitting time of $y \in Z_n$ with respect to $(X_{t,n},P_n)$. Note that $Q_n$ is $\RR[\Sn_n]$-linear. As before, we may therefore make sense of $Q_n(\Ob_n)$ whenever $\Ob$ is a stable orbit of $Z_\dt \times Z_\dt$.
\end{definition}

\emph{For the remainder of this section, we fix a finitely generated $\FI$--set $Z_\dt$ and a connected Markov chain $(X_{t,\dt},P)$ on $Z_\dt$}.

\begin{lemma}\label{matrixeqn}
For $n \gg 0$, let $x, y \in Z_n$ be distinct. Then one has,
\[
Q_n(x,y) = 1+ \sum_{[z,y]} P^y_n([x,y],[z,y])Q_n([z,y])
\]
where the sum is over all $y$-roofed orbits, and $P^y_n$ is the matrix as in Definition \ref{orbitgraph}
\end{lemma}

\begin{proof}
Using Lemma \ref{hitrecursion} we have,
\[
Q_n(x,y) = 1 + \sum_{z \neq y} P_n(x,z)Q(z,y)
\]
Using the fact that $Q(z,y) = Q(w,y)$ whenever $[z,y] = [w,y]$, we may gather like terms in this expression to find
\[
Q(x,y) = 1 + \sum_{[z,y] \neq [y,y]} \left( \sum_w P_n(x,w) \right) Q([z,y]),
\]
where the outer sum is over $y$-roofed orbits $[z,y] \neq [y,y]$, and the inner sum is over all  $w \in Z_n$ such that $[z,y] = [w,y]$.

Next we simplify the sum $\sum_w P_n(x,w)$ appearing above. We have that $P_n(x,w) = P_n(x,w')$ whenever $[x,w] = [x,w']$, and so we have
\[
\sum_w P_n(x,w) = \sum_{[x,w]} |\{w' \mid [x,w] = [x,w'], [w',y] = [z,y]\}|  P_n([x,w]),
\]
where the sum is over all $x$-rooted orbits $[x,w]$. The proof of Lemma \ref{roofedorbit} immediately implies our conclusion.
\end{proof}

Lemma \ref{matrixeqn} allows us to approach computing hitting times from a new perspective. Recall that we defined the vector space $V^y_{Z_\dt}$, the $K$-linearization of the $y$-roofed stable orbits of $Z_\dt \times Z_\dt$. Let $\widetilde{V^y_{Z_\dt}}$ be the subspace of $V^y_{Z_\dt}$ spanned by the orbits not equal to the identity orbit. Lemma \ref{matrixeqn} implies that computing the hitting times into $y$ is equivalent to solving the $K$-linear equation
\begin{eqnarray}
Q = \mathbf{1} + P^y|_{\widetilde{V^y_{Z_\dt}}} \cdot Q,\label{solvethis}
\end{eqnarray}
where $\mathbf{1} = \sum_{[z,y] \neq [y,y]} [z,y]$, in the vector space $\widetilde{V^y_{Z_\dt}}$.

\begin{proposition}
The equation (\ref{solvethis}) admits a unique solution $Q$ in  $\widetilde{V^y_{Z_\dt}}$.
\end{proposition}

\begin{proof}
Rewriting (\ref{solvethis}) we find that we must show that the equation
\[
(id - P^y|_{\widetilde{V^y_{Z_\dt}}})Q = \mathbf{1}
\]
has a unique solution. This will follow from showing that $P^y|_{\widetilde{V^y_{Z_\dt}}}$ does not have 1 as an eigenvalue. To see that, recall that for each $n$ $P^y_n$ was the transition matrix of a connected Markov chain. From Remark \ref{righteigen} we know that every eigenvector for the eigenvalue 1 of $P^y_n$ is necessarily supported on every $y$-rooted orbit, including the identity orbit. This implies that $P^y|_{\widetilde{V^y_{Z_\dt}}}$ cannot have 1 as an eigenvalue.
\end{proof}

This proposition is the key fact that we need to prove the main theorem of this section.

\begin{theorem}\label{hittingthm}
There exists a virtual relation $Q$ such that for $n \gg 0$, the hitting time matrix $Q_n:\RR Z_n \rightarrow \RR Z_n$ is the specialization of $Q$ to $n$. In particular, if $\Ob$ is a stable orbit of $Z_\dt \times Z_\dt$, then the map
\[
n \mapsto Q_n(\Ob_n)
\]
agrees with a rational function for all $n \gg 0$.
\end{theorem}

\begin{example}
Let $Z_n = [n]$, and let $(X_{t,n},P_n)$ represent the simple random walk on the complete graph. Then it is easily shown that
\[
Q_n(x,y) = \begin{cases} n-1 &\text{ if $x \neq y$}\\ 0 &\text{ otherwise.}\end{cases}
\]

We may also take $Z_n = [n] \sqcup [n]$, and have $(X_{t,n},P_n)$ represent the simple random walk on the complete bipartite graph $K_{n,n}$. Then,
\[
Q_n(x,y) = \begin{cases} 2n-1 &\text{ if $x$ and $y$ are in different parts}\\ 2n &\text{ if $x \neq y$ are in the same part}\\ 0 &\text{ if $x = y$.}\end{cases}
\]
\end{example}

We would next like to use our previously described connections between hitting times and Green's functions to conclude asymptotic information about Green's functions of $\FI$-graphs. Before we can do this, we need to establish some notation.

\begin{definition}
A \textbf{weighted $\FI$-graph} is a pair of an $\FI$-graph $G_\dt$, and a virtual relation $w$ on $V(G_\dt)$, such that for all $n\gg 0$, $(G_n,w_n)$ is a weighted graph. We say that a weighted $\FI$-graph is \textbf{stochastic} if $w_n$ is stochastic for all $n\gg 0$.
\end{definition}

Given a weighted $\FI$-graph, there is an obvious way we might make sense of the normalized Laplacians $\Lap_n$ and Green's functions $\Gre_n$ for $n \gg 0$. This allows us to conclude the following.

\begin{corollary}
Let $G_\dt$ be a weighted graph with weight function $w$. Then if $\Ob$ is a stable orbit of $Z_\dt \times Z_\dt$, the map
\[
n \mapsto \Gre_n(\Ob_n)
\]
agrees with a function which is algebraic over $K$. More specifically, it agrees with a function involving rational functions and possibly square roots of rational functions.
\end{corollary}

\begin{proof}
This follows from Theorems \ref{hittingthm} and \ref{hittinggreen} and Remark \ref{allstoch}.
\end{proof}

\begin{example}
For all examples we will consider weighted $\FI$-graphs $G_\dt$ with,
\[
w_{n,x,y} = \begin{cases} 1 &\text{ if $\{x,y\} \in E(G_n)$}\\ 0 &\text{ otherwise.}\end{cases}
\]
In the case where $G_\dt = K_\dt$ is the complete graph, Xu and Yau have computed \cite[Example 3.7]{XY}
\[
\Gre_n(x,y) = \begin{cases} \frac{-1}{n^2} &\text{ if $x \neq y$}\\ \frac{n-1}{n^2} &\text{ otherwise.}\end{cases}
\]

Xu and Yau have also computed the Green's function for the star graphs $K_{\star,1}$ \cite[Example 3.8]{XY}. Denote the center of $K_{n,1}$ by $c$, and write $x,y$ for two distinct non-central vertices. Then,
\begin{eqnarray*}
\Gre_n(c,c) = \frac{1}{4(n-1)}, \hspace{.25cm} \Gre_n(x,x) = \frac{4n-7}{4(n-1)},\\
\Gre_n(x,y) = \frac{-3}{4(n-1)}, \hspace{.25cm} \Gre_n(c,x) = \frac{-1}{4(n-1)}.
\end{eqnarray*}
\end{example}

\subsection{Higher moments of hitting time random variables}

In this section we focus on what can be said about the higher moments of hitting time random variables. While our main result will hold for all of the higher moments, for concreteness and simplicity we will limit our exposition to the variance of these random variables. See Remark \ref{highermoments} for an explanation on how the proof of Theorem \ref{varrational} can be generalized to higher moments.

\begin{definition}
Let $X$ denote a random variable on some probability space. We write $\Var(X)$ to denote the \textbf{variance} of $X$, 
\[
\Var(X) = \mathbb{E}(X^2)-(\mathbb{E}(X))^2.
\]

If $Z_\dt$ is a finitely generated $\FI$-set, and $(X_{\dt,t},P_\dt)$ is a Markov chain on $Z_\dt$, then for any orbit of pairs $\Ob$ on $Z_\dt$ and any $n \gg 0$, one may make sense of the quantity
\[
\Var(\tau_{\Ob_n}).
\]
Our objective will be to develop an understanding of the function $n \mapsto \Var(\tau_{\Ob_n})$.
\end{definition}

Our main strategy will be similar to that which we applied to the first moment of the hitting time random variable. Namely, we apply recursive methods and reduce the problem to a finite one on orbits. To this end we recall the following fact.

\begin{theorem}[The Law of Total Variance]
Let $X,Y$ be a random variables on some probability space, and assume that the variance of $X$ is finite. Then,
\[
\Var(X) = \mathbb{E}(\Var(X \mid Y)) + \Var(\mathbb{E}(X \mid Y))
\]
\end{theorem}

In the context of hitting times, the law of total variance simplifies quite significantly.

\begin{lemma}
Let $(X_t,P)$ denote a Markov chain with state space $Z$, and let $a,b \in Z$ be distinct. Then,
\begin{eqnarray}
\Var(\tau_{a,b}) = \Var(\mathbb{E}(\tau_{\dt,b})) +  \sum_{y} P(a,y)\cdot \Var(\tau_{y,b}), \label{recurrence}
\end{eqnarray}
where $\mathbb{E}(\tau_{\dt,b})$ is the expected hitting time of $b$ from a random neighbor $y$ of $a$, which is a random quantity which depends on $y$, and the sum is over neighbors $y$ of $a$.
\end{lemma}

\begin{proof}
We apply the law of total variance. We take $X$ to be the hitting time $\tau_{a,b}$, and $Y$ to be the (random) choice of the first step in the random walk. Looking first at the term $\Var(\mathbb{E}(\tau_{a,b}\mid Y))$ we see by definition, for any neighbor $y$ of $a$,
\begin{eqnarray*}
\mathbb{E}(\tau_{a,b} \mid Y)(y) &=& \sum_{i \geq 0} i \cdot \mathbb{P}(\mathbb{E}(\tau_{a,b} \mid Y = y) = i)\\
						&=& \sum_{i \geq 0} i \cdot \mathbb{P}(\mathbb{E}(\tau_{y,b} = i-1)\\
						&=& \mathbb{E}(\tau_{y,b}) + 1.
\end{eqnarray*}
Variance is unchanged by constant shifts, so $\Var(\mathbb{E}(\tau_{a,b}\mid Y)) = \Var(\mathbb{E}(\tau_{\dt,b}))$, as desired.

Next we consider $\mathbb{E}(\Var(X \mid Y))$. By definition we have,
\begin{eqnarray*}
\mathbb{E}(\Var(\tau_{a,b} \mid Y)) &=& \sum_y \Var(\tau_{a,b} \mid Y = y) \cdot P(a,y)\\
						   &=& \sum_y \Var(\tau_{y,b} + 1) \cdot P(a,y)\\
						   &=& \sum_y \Var(\tau_{y,b}) \cdot P(a,y).
\end{eqnarray*}
Note that we once again used that variance is unchanged by constant shifts. This concludes the proof.
\end{proof}

This is all we need to prove the main theorem of this section.

\begin{theorem}\label{varrational}
Let $Z_\dt$ denote a finitely generated $\FI$--set, and let $(X_{\dt,t},P_\dt)$ denote a Markov chain on $Z_\dt$. Then there exists a virtual relation $\Var$ such that for $n \gg 0$, the matrix $\Var_n:\RR Z_n \rightarrow \RR Z_n$ with entries $\Var_n(a,b) = \Var(\tau_{a,b})$ is the specialization of $\Var$ to $n$. In particular, if $\Ob$ is a stable orbit of $Z_\dt \times Z_\dt$, then the map
\[
n \mapsto \Var_n(\Ob_n)
\]
agrees with a rational function for all $n \gg 0$.
\end{theorem}

\begin{proof}
As with the proof of Theorem \ref{hittingthm}, we work on the $\RR(n)$-vector space of roofed orbits with a fixed choice of roof. Looking at equation (\ref{recurrence}) we see that $\Var_n(a,b)$ satisfies a linear equation almost identical to $Q_n(a,b)$. The techniques used in the proof of Theorem \ref{hittingthm} will therefore apply in identical fashion in this setting as well, so long as we know that the term $\Var(\mathbb{E}(\tau_{\dt,b}))$ can be expressed as a rational function in $n$. This follows from Theorem \ref{hittingthm}.
\end{proof}

\begin{remark}\label{highermoments}
The conclusion of Theorem \ref{varrational} remains true when one replaces the variance with any of the higher moments. The proof is largely the same, and proceeds in two steps. The Law of Total Variance is a special case of the Law of Total Cumulance. As above, this theorem tells us that the $i$--th cumulant can be expressed as the conditional expected value of the $i$-th cumulant, plus a term involving only lower moments. An induction argument therefore implies that the $i$-th cumulant agrees with a rational function in $n$. Following this, one uses standard facts to write the $i$--th moment as a polynomial combination of cumulants.
\end{remark} 

\begin{remark}
It would be interesting to be able to bound the relative sizes of the moments --- for instance, are there conditions under which the hitting times have small standard deviation compared to their expected value? Our techniques do not appear to give such bounds.
\end{remark}

\end{document}